\documentclass{amsart}
\usepackage[margin=2cm]{geometry}
\usepackage{graphicx} % Required for inserting images
\usepackage{enumerate}
\usepackage{amssymb}
\usepackage{amsaddr}
\usepackage[utf8]{inputenc}
\usepackage{amsmath}
\usepackage{amsfonts}
\usepackage{amsthm}

\usepackage[hypertexnames=false]{hyperref}
\usepackage{resizegather}
\usepackage{mathtools}
\usepackage{arydshln}
\usepackage{algorithm}
\usepackage{algpseudocode}
\usepackage{float}% http://ctan.org/pkg/float
\usepackage{xcolor}
\usepackage{cleveref}

\newtheorem{thm}{Theorem}[section]
\newtheorem{cor}[thm]{Corollary}

\newenvironment{claimproof}[1][Proof of Claim]{\begin{proof}[#1]}{\end{proof}}
\newtheorem{lem}[thm]{Lemma}
\newtheorem{prop}[thm]{Proposition}
\newtheorem{prob}[thm]{Problem}

\newtheorem{defn}[thm]{Definition}
\newtheorem{examp}[thm]{Example}
\newtheorem{notn}[thm]{Notation}
\newtheorem{obs}[thm]{Observation}
\theoremstyle{remark}
\newtheorem{rem}[thm]{Remark}
\newtheorem{claim}[thm]{Claim}

\numberwithin{equation}{section}

\crefname{lemma}{Lemma}{Lemmas}
\Crefname{lemma}{Lemma}{Lemmas}

\algrenewcommand\algorithmicrequire{\textbf{Input:}}
\algrenewcommand\algorithmicensure{\textbf{Output:}}
\algnewcommand{\return}{\textbf{return }}
\DeclareMathOperator*{\argmin}{argmin}
\DeclareMathOperator*{\argmax}{argmax}
\newcommand{\Mod}[1]{\ (\mathrm{mod}\ #1)}

\numberwithin{equation}{section}

\newcommand{\PP}{\mathcal P}

\newcommand{\snk}{s^{n,k}}

\newcommand{\slack}{\ensuremath{\textrm{slack}}}
\newcommand{\ms}{\ensuremath{\textrm{maxSum}}}
\newcommand{\N}{\mathbb{N}}      % The natural numbers
\newcommand{\NZ}{\mathbb{N}_0}   % The natural numbers + {0}
\newcommand{\R}{\mathbb{R}}      % The natural numbers
\newcommand{\iterI}{{\textrm{Iter}}_1}
\newcommand{\iterII}{{\textrm{Iter}}_2}

\title{Solvable and unsolvable instances of the equal sum partition problem}

\date{}

\author{Shlomo Hoory \and Dani Kotlar}
\address{Department of Computer Science, Tel-Hai University, Kiryat Shmona, Israel}

\email{hooryshl@telhai.ac.il}
\email{dannykot@telhai.ac.il}
\thanks{*Corresponding author: Shlomo Hoory (hooryshl@telhai.ac.il)}
\begin{document}

\maketitle

\begin{abstract}
We consider the equal sum partition problem, motivated by distance magic graph labeling:
Given $n,k \in \N$ such that $k\, | \sum_{i=1}^ni$ and a partition $p_1+\cdots+p_k=n$, when is it possible to find a partition of the set $\{1,2,\ldots,n\}$ into $k$ subsets of sizes $p_1,\dots,p_k$, such that the element sum in each subset is the same? 

A known necessary condition is the \emph{slack condition}, requiring that for all $j$, placing the largest possible elements in the $j$ smallest sets yields a total sum that is at least what is needed. However, this condition is not sufficient, and known counterexamples exist. 

This work clarifies the boundary between solvable and unsolvable instances of the problem. 
We extend the list of unsolvable problem instances satisfying the slack condition by exhibiting infinite families where the $n/k$ ratio is any rational number in the interval $(2,\frac{24}{7})$,
and a new criterion for unsolvability.
Furthermore, we show that the slack condition is natural, as it is both necessary and sufficient for the fractional relaxation of the problem. Based on this result, we prove that the problem is solvable for the class of linear partitions, where $k$ is fixed, $p_1,\ldots,p_k$ grow linearly with $n$, and where the slack condition holds in a strong sense. We do this by applying a randomized rounding algorithm to a solution of the fractional relaxation of the problem and proving that the algorithm has an exponentially small failure probability.

\end{abstract}

\section{Introduction}

\subsection{Background}

A \emph{distance magic labeling} \cite{miller2003} (or \emph{sigma labeling} \cite{beena2009}) of a simple undirected graph with $n$ vertices is a bijective assignment of the numbers in $[n]=\{1,2,\ldots,n\}$ to its vertices, such that the sum of the values assigned to the neighbors of each vertex is the same.

Challenges involving distance-magic labelings, such as deciding whether one exists, and constructing one when it does, are closely tied to questions about partitioning the set $[n]$. A classic example is the following: Given positive integers $n$ and $k$ with $k$ dividing $\sum_{i=1}^n i$, can $[n]$ be partitioned into $k$ subsets whose element sums are all equal? Straight and Schillo~\cite{straight1979} showed that the answer is yes. But what happens when the sizes of these subsets are prescribed? This is the subject of our work.

\begin{prob}[\cite{anholcer2016spectra} Problem 6.3 and \cite{cichacz2018constant}, Problem 1.1]\label{prob:nk_partition}
Let $n,k$ and $p_1,\dots,p_k$ be positive integers such that $p_1+\cdots+p_k=n$ and $k$ divides $\sum_{i=1}^ni$. When is it possible to find a partition of the set $[n]$ into $k$ subsets of sizes $p_1,\dots,p_k$, respectively, such that the sums of the elements in each subset are equal?
\end{prob}

Problem~\ref{prob:nk_partition} was shown to be equivalent to problems of finding a distance magic labeling of complete multipartite graphs (Miller, Rodger and Simanjuntak \cite{miller2003}, Cichacz and G{\H{o}}rlich \cite{cichacz2018constant}), and of graphs obtained by replacing each vertex in $C_k$ by a clique, with a slightly modified definition of distance magic labeling (Anholcer, Cichacz and Peterin \cite{anholcer2016spectra}). For more results and surveys on distance magic labeling, see \cite{anholcer2015distance}, \cite{arumugam2012distance}, and \cite{gallian2018dynamic}. 

A clearly necessary condition for the existence of a partition as described in Problem~\ref{prob:nk_partition} is that, for every $j$, if we place the largest available elements into the $j$ smallest parts, the resulting total must be at least as large as the required sum, possibly with some additional ``slack''. We refer to this requirement as \emph{the slack condition} (see Definition~\ref{def:slack_condition}).

It was shown by Miller, Rodger and Simanjuntak \cite{miller2003} and Beena \cite{beena2009} that the slack condition is sufficient when $k=2,3$. Kotlar \cite{kotlar16} extended the result to $k=4$ and dared to conjecture (as did Arumugam, Froncek, Dalibor and Kamatchi \cite{arumugam2012distance}) that this holds in general. However, Ebrahem Hoory and Kotlar \cite{ebrahem2024} exhaustively checked all instances of Problem~\ref{prob:nk_partition} with $n\le200$, and identified two counterexample structures. Examples~\ref{ex:39:13} and \ref{ex:80:30} illustrate these structures.

In this paper, following the approach taken in two earlier works \cite{kotlar16, ebrahem2024}, we study Problem~\ref{prob:nk_partition} independently of the graph‑labeling questions related to it, as we think the problem is intrinsically interesting.

%%%%%%%%%%%%%%%%%%%%%%%%%%%%%%%%%%%%%%%%%%%%%%%%%%%%%%%%%%%%%%%%%%%%%%%%%%%%%%%%%%%%%%%%

\subsection{Definitions and notation}

For a positive integer $n$ a \emph{composition} of $n$ is a sequence of positive integers $P=[p_1,p_2,\ldots, p_k]$ such that $\sum_{i=1}^k p_i = n$.
Given a composition of $n$ so that $k$ divides $\sum_{i=1}^ni$, we would like to know whether it is possible to find a partition of the set $[n]$ into $k$ subsets of sizes $p_1,\dots,p_k$, respectively, such that the sum of the elements in each subset is the same.
There is a necessary condition for this. To introduce it, we need some more definitions:

\begin{notn}
    Given a composition $[p_1,p_2,\ldots, p_k]$ of $n$, where $k$ divides $\sum_{i=1}^ni$, we denote the target sum for partitioning $[n]$ into $k$ equal sum parts by
    \[
    s^{n,k}=\frac{n(n+1)}{2k}.
    \]
\end{notn}

\begin{defn}\label{def:slack}
    Let $P=[p_1,p_2,\ldots, p_k]$ be a non-descending composition of a positive integer $n$. For $j=1,\ldots,k$ let $P_j=\sum_{i=1}^jp_i$. For $j=1,\ldots,k-1$ let 
    $$\textrm{slack}_j(P) = \sum_{i=1}^{P_j} (n-i+1)- js^{n,k}.$$ 
    We denote the slack of $P$ by
    $$\textrm{slack}(P)=\min_{1\le j<k}\textrm{slack}_j(P).$$

\end{defn}

\begin{rem}
    We do not include $\textrm{slack}_k(P)$ in the minimum expression, as it is always 0.
\end{rem}

\begin{obs}[\cite{anholcer2016spectra}]
    Let $P=[p_1,p_2,\ldots, p_k]$ be a non-descending composition of a positive integer $n$. If it is possible to partition the set $[n]$ into $k$ equal-sum subsets of sizes $p_1,\dots,p_k$, then $\textrm{slack}(P)\ge 0$.
\end{obs}

\begin{defn}\label{def:slack_condition}
    We call the condition ${\textrm slack}(P)\ge0$ \emph{the slack condition}. If this condition holds for $P$, we say that $P$ satisfies the slack condition.
\end{defn}

As satisfying the slack condition is a necessary condition for the existence of a partition of the set $[n]$ into $k$ equal-sum subsets of sizes given by $P$, we make it part of the definition:

\begin{defn}
    A triple $(n, k, P)$ is an instance of the Equal Sum Partition Problem (ESPP instance, for short), if $n$ and $k$ are positive integers such that $k$ divides $\frac{n(n+1)}{2}$, and $P=[p_1,p_2,\ldots, p_k]$ is a non-descending composition of $n$ satisfying the slack condition.     
    We say that an ESPP-instance $(n, k, P)$ is solvable if it is possible to partition the set $[n]$ into subsets of sizes $p_1,p_2,\ldots, p_k$ respectively, so that the sum of the elements in each set is the same. 
\end{defn}

\begin{rem}
    Since the composition $P$ in an ESPP $(n,k,P)$ contains the information about $n$ and $k$, we shall sometimes use $P$ instead of the triple $(n,k,P)$ to refer to an ESPP instance. For example, we will say that $P$ is solvable, meaning that the ESPP instance $(n,k,P)$ is solvable.
\end{rem}

\begin{notn}
    Let $(n,k,P)$ be an ESPP instance. The notation $P=[q_1^{e_1},q_2^{e_2},\ldots, q_t^{e_t}]$ means that $q_1<q_2<\cdots<q_t$ and for $i=1,\ldots,t$, $P$ has $e_i$ parts of size $q_i$. Each $q_i^{e_i}$ is a block, or the $q_i$-block. 
\end{notn}

\subsection{Our results}

In this work, we expand the boundary of what is known about ESPP instances in both directions: non-solvable and solvable. 
For the non‑solvable direction, we show a new criterion for unsolvability of ESPP instances (\Cref{thm:criterion3}).
Additionally, based on a criterion from \cite{ebrahem2024}, infinitely more unsolvable instances are provided by the following theorem:

\begin{thm}\label{thm:inifinite_family}
For any rational number $2<a<\frac{24}{7}$, there is an infinite family of ESPP instances $(n,k,P_k)$, where $a=n/k$, that are not solvable.
\end{thm}

For the solvable direction, we analyze linear families $\PP_{\alpha_1,\ldots,\alpha_k}$, each consisting of instances of the form $(n,k,P_n=[\alpha _1n,\ldots ,\alpha _kn])$, where $[\alpha_1,\ldots ,\alpha_k]$ is a non-descending sequence of positive rational numbers summing to 1, and satisfying a strengthened version of the slack condition (Definition~\ref{def:slack_alphas}). Using a randomized rounding algorithm (Algorithm~\ref{alg:linear_partition}), we construct an equal‑sum partition for any such instance, and our main result is:

\begin{thm}\label{thm:alg_solves_linear_partition_whp}
    Let $\PP_{\alpha_1,\ldots,\alpha_k}$ be a linear partition problem family. 
    Then, there exist $\delta>0$ and $N \in \N$ so that for every $n \ge N$, the probability that Algorithm~\ref{alg:linear_partition} fails to solve $(n,k,P_n) \in \PP_{\alpha_1,\ldots,\alpha_k}$ is at most $e^{-\delta n}$.
\end{thm}

The randomized rounding algorithm is based on our solution (\Cref{alg:fluid_mixing_solution_rec}) to a version of the fluid mixing problem in operations research. This problem arises as a fractional version of the equal-sum partition problem. We also define an analogue of the slack property, which coincides with the usual slack in the integral case. We do not describe this problem here, as its definition is elaborate, but we state the main result in this context: 

\begin{thm}\label{thm:fractional}
    An instance of the fluid mixing problem is solvable if and only if it satisfies the fractional slack condition.
\end{thm}

%%%%%%%%%%%%%%%%%%%%%%%%%%%%%%%%%%%%%%%%%%%%%%%%%%%%%%%%%%%%%%%%%%%%%%%%%%%%%%%%%%%%%%%%%%%%%%%%%%%%%%%%%%%

\section{Non-solvable ESPP instances}\label{sec:counterexamples}
In this section, we study families of non-solvable ESPP instances. In the first part, we introduce the useful notion of an \emph{incomplete ESPP instance}. In the second part, we introduce a new criterion for unsolvability. In the third part, we show that a criterion from a previous work (\cite{ebrahem2024}, Criterion 1) defines infinite families of such instances.

\subsection{Incomplete ESPP instances} 
It is possible to identify the unsolvability of some ESPP instances by looking at a prefix of the corresponding nondescending composition of $n$. For this, we need some new notions.

\begin{defn}\label{def:incomplete_composition}
    An incomplete non-descending composition of a positive integer $n$ is a non-descending sequence of positive integers $Q=[p_1,\ldots,p_l]$ such that $k<l$ and 
    \begin{equation}\label{def:inc_part_cond}
        n-\sum_{i=1}^l p_i \geq (k-l)p_l
    \end{equation}
    A non-descending composition $P$ of $n$ is a completion of an incomplete non-descending composition $Q$ if $Q$ is a prefix of $P$.
\end{defn}

\begin{prop}[\cite{ebrahem2024}, Lemma 6.3]\label{lem:incomplete_partition_completion}
    Any incomplete non-descending composition of a positive integer $n$ has a completion.
\end{prop}

\begin{defn}\label{def:partial_slack}
    Let $n$ and $k$ be positive integers such that $\snk$ is an integer, and let $Q=[p_1,\ldots,p_l]$ be an incomplete non-descending composition of $n$. We define $\textrm{slack}_j(Q)$, for $j\le l$, as in Definition~\ref{def:slack}, and $\textrm{slack}(Q)=\min_{1\le j\le l}\textrm{slack}_j(Q)$.
\end{defn}
Note that while the minimum expression for $\slack(P)$ when $P$ is a non-descending composition of $n$ omits the last index $j=k$, the minimum for $\slack(Q)$, when $Q$ is an incomplete non-descending composition, includes $l$. 

A slightly different version of the following lemma appears in  \cite{ebrahem2024} (Lemma 6.4) for a non-descending composition. We state it here for an incomplete non-descending composition. We omit the proof, as it is the same as the one in \cite{ebrahem2024}. 

\begin{lem}\label{lem:slack_at_end_of_block}
In the setup of Definition~\ref{def:partial_slack}, the value of $\slack(Q)=\min_{1\le j\le l}\slack_j(Q)$ is obtained at $j$ such that $p_j$ is the last part in $Q$ of its size. 
\end{lem}

\begin{defn}\label{def:inc_part}
    An \emph{incomplete instance of an equal sum partition problem} (incomplete ESPP instance) is a triple $(n,k,Q)$, where $n$ and $k$ are positive integers such that $\snk$ is an integer, $Q$ is a non-descending incomplete composition of $n$, and $\textrm{slack}(Q)\ge0$.  
\end{defn}

\begin{defn}
    Given an incomplete ESPP instance $I=(n,k,Q)$, a completion of $I$ to an ESPP instance is an ESPP instance $(n,k,\PP)$, such that $P$ is a completion of $Q$. 
\end{defn}

Note that by Proposition~\ref{lem:incomplete_partition_completion}, any incomplete ESPP instance has a completion.

We quote {\bf Criterion 1} from \cite{ebrahem2024}:

\begin{prop}[\cite{ebrahem2024}, Theorem 4.2]\label{prop:criterion1}
    Let $I=(n,k,[2^e, d^f])$ be an incomplete ESPP instance, where $e,f>0$ and $d \ge 3$. Let $C = \{x\in [n]:x \ge \snk-n\}$, $c = \snk-n$, and $h = |C| - 2e = 2(n-e) - \snk + 1$. If 
    \begin{enumerate}
        \item [(a)] $f>h$, and
        \item [(b)] $\sum_{i=c-d(f-h)}^{c-1}i < (f-h)\snk$,
    \end{enumerate}
    then any completion of $I$ to an ESPP instance is not solvable.
\end{prop}

The idea of this criterion is that any attempt to find a partition of $[n]$ with equal sums runs out of `big' numbers, namely elements of the set $C$, so that it is impossible to satisfy the remaining parts because of what we loosely call the ``slack without $C$ condition'' (Condition (b) in Proposition~\ref{prop:criterion1}). This type of argument is illustrated in the next example, which introduces the unsolvable ESPP instance with the smallest values of $n$ and $k$.

\begin{examp}[\cite{ebrahem2024}]\label{ex:39:13}
Any completion of the incomplete ESPP instance $I=(39,13,[2^9, 3^2])$ is not solvable.
\end{examp}

\begin{proof}[Reason]
It can be verified that $I$ is indeed an incomplete ESPP instance and that $s^{39,13}=60$. Using the notation of Proposition~\ref{prop:criterion1}, we have $f=2$, $c=21$ and $h=1$. So, clearly (a) holds. The sum in (b) is $\sum_{i=18}^{20}i=57<60$. By Proposition~\ref{prop:criterion1}, the claim holds.  We now explain why this is true by a direct argument: Let $P=[2^9, 3^2, p,q]$ be a completion of $[2^9, 3^2]$.  Suppose that there is a partition of $[39]$ containing nine parts of size 2 and two parts of size 3 so that all its parts' sums are 60. The parts of size 2 must contain the numbers $21,22,\ldots,39$, excluding 30, in pairs of sum 60. On the other hand, any part of size $3$ must contain at least one number greater than $20$. Since there are two parts of size 3 and only one number greater than $20$ left, such a partition is impossible. 
\end{proof}

The unsolvability of the incomplete ESPP instance in the next example follows from {\bf Criterion 2} in \cite{ebrahem2024}. Rather than quoting the criterion, we give an example-specific explanation for unsolvability.  

\begin{examp}[\cite{ebrahem2024}]\label{ex:80:30}
Any completion of the incomplete ESPP instance $I=(80,30,[2^{25}, 3^2])$ is not solvable.\end{examp}

\begin{proof}[Reason]
    It can be verified that $I$ is indeed an incomplete ESPP instance and that  $s^{80,30}=108$. Let $P=[2^{25}, 3^2, p,q,r]$ be a completion of $[2^{25}, 3^2]$. 
    Suppose that there is an equal sum partition of $[80]$ with parts of sizes corresponding to the elements of $P$. Then, each of the 25 pairs must be assigned two elements from the set $C=\{c,\ldots,80\}$, where $c=28$ and $|C|=53$. Therefore, after assigning the 25 pairs, only three $C$-elements $x,m,y$ are left, where $x < m < y$, and where $m=54$ is the middle element of $C$.
    Since $m + (c-1) + (c-2) = 107 < s^{80,30}=108$, we conclude that the block $3^2$ consumes these 3 elements, and that the subsequent $4$-tuple must be satisfied with elements smaller than $c$.
    However, $\sum_{i=c-4}^{c-1} i = 102 < s^{80,30}$, contradicting the assumption that all parts have the same sum.
\end{proof}

\subsection{A criterion for unsolvability}\label{sec:criterion}
We introduce a new criterion for the unsolvability of an ESPP instance. 
% We call it Criterion 3 (Criterions 1 and 2 appear in \cite{ebrahem2024}).

\begin{notn}
    For a set of numbers $A$ and a positive integer $l\le |A|$, we denote by $\ms(l, A)$ the sum of the $l$ largest elements in $A$.
    For two integers $a<b$ we write $[a,b]$ for the set $\{a,a+1,\ldots,b\}$.
\end{notn}

\begin{thm}\label{thm:criterion3}
    Let $I=(n,k,P=[2^d,\ldots])$ be an ESPP instance. Let $u=2n-\snk+1-2d$, $e = \lceil(u+1)/2\rceil$ and $m=\snk-n-1$.
    Then, $I$ is not solvable if one of the following conditions holds for some $1 \le i \le k-d-u$:
    \begin{enumerate}
        \item[(I)]  
            $\ms(\sum_{i'=1}^i p_{d+u+i'},[m]) < i \snk$
        \item[(II)] 
            $\ms(\sum_{i'=1}^i p_{d+u+i'},[m]) = i \snk$ 
            and $\ms(p_{d+e}-1,[m-\sum_{i'=1}^i p_{d+u+i'}])<\snk/2$.
    \end{enumerate}    
\end{thm}

\begin{proof}
    Assume for contradiction that (I) or (II) holds for some $1 \le i \le k-d-u$, and there is a partition $\PP$ of $[n]$ such that the sum in each part is $\snk$. 
    Assume $\PP=[A_1,A_2,\ldots,A_k]$, where each $A_i$ has size $p_i$.
    Let $C=[\snk-n,n]$ and let $C'$ be the subset of $C$ consisting of those elements that do not appear in the pairs. Since the $d$ pairs contain only elements of $C$, we have $|C'|=u$.
    Let $J$ be the set of indices in $[d+1, d+u+i]$ such that $A_j\cap C'=\emptyset$. 
    Since $|\{j\in [d+1,d+u+i]:A_j\cap C'\ne \emptyset\}|\le |C'|=u$, we have $|J|\ge (u+i)-u=i$. Let $J_1\subseteq J$ be any subset of size $i$.
    Denoting by $S(A)$ the sum of the elements in the set $A$, we have 
    \begin{equation}\label{eq:sum_of_isnk}
    i\snk=\sum_{j\in J_1}S(A_j)\le\ms(\sum_{j\in J_1}p_j,[m])\le \ms(\sum_{i'=1}^i p_{d+u+i'},[m]) \le i \snk     
    \end{equation}
    The first inequality holds since the $A_j$, $j\in J_1$, consist only of elements of $[m]$, but not necessarily the biggest ones. The second inequality follows from the fact that the $p_t$ are non-descending. The third inequality combines (I) and (II). 
    If (I) holds, the third inequality in \eqref{eq:sum_of_isnk} is strict, and we have a contradiction.
    If (II) holds, by \eqref{eq:sum_of_isnk}, we have that all subsets of size $i$ of $J$ are the same, so $|J|=i$, and we may assume that $J=[d+u+1,d+u+i]$. So, the sets $A_{d+u+1},\ldots, A_{d+u+i}$ contain the $\sum_{i'=1}^i p_{d+u+i'}$ largest elements of $[m]$ and each of the sets $A_{d+1},\ldots,A_{d+u}$ contains exactly one element of $C'$, and all the rest of the elements in these sets are from $[m-\sum_{i'=1}^i p_{d+u+i'}]$.
    
    Since the set $C'$ consists of pairs of elements $x,\snk-x$, where $x<\snk/2$, and possibly $\snk/2$ itself (if $\snk$ is even), at least half the elements in $C'$ are less than or equal to $\snk/2$, and thus, at least one of the sets $A_{d+1},\ldots, A_{d+e}$ contains, as its $C'$ element, a number less than $\snk/2$. Suppose it is $A_{u+l}$, $l\le e$. Let $s_l$ be the sum of $A_{u+l}$'s non-$C'$ elements. Then, $s_l>\snk/2$. On the other hand, $s_l\le \ms(p_{d+l}-1,[m-\sum_{i'=1}^i p_{d+u+i'}])\le \ms(p_{d+e}-1,[m-\sum_{i'=1}^i p_{d+u+i'}])<\snk/2$, by (II) - a contradiction.
\end{proof}

Case (I) of this criterion is an extension of Criterion 1 in \cite{ebrahem2024} (Proposition~\ref{prop:criterion1} here) and is in fact the version that was used for the search in \cite{ebrahem2024}, although it is not stated there in this form. 
This criterion also applies to incomplete ESPP instances, as long as the index $d+u+i$ falls within the corresponding incomplete composition. 

A computerized exhaustive search with $n$ up to 500 found 17,050 incomplete ESPP instances that satisfy the criterion in Theorem~\ref{thm:criterion3}. Of these 17,050 instances, only 7 belong to Case (II). The two instances of Case (II) with the lowest value of $n$ are $(208,76, [2^{64}, 3^2, 4^4])$ and $(299,115, [2^{103}, 4^2, 5^6])$. 
In both instances, the value of $i$ is such that the index $d+u+i$ reaches exactly the end of the incomplete composition.

\subsection{Infinite families of non-solvable ESPP instances}\label{subsec:families}
 
In this subsection, we use Proposition~\ref{prop:criterion1} to show that for each $2<a<\frac{24}{7}$ there are infinitely many unsolvable ESPP instances with $n/k = a$. These instances have the form $(ak,k,[2^{uk}, d^{vk},\ldots])$, where $uk,vk$ and $d\ge 3$ are positive integers, and $0 < u,v < 1$.  

We restate Theorem~\ref{thm:inifinite_family}:

\medskip

\noindent\textbf{Theorem \ref{thm:inifinite_family}.}
\emph{For any rational number $2<a<\frac{24}{7}$, there is an infinite family of ESPP instances $(n,k,P_k)$, where $a=n/k$, that are not solvable.}

\medskip

Before delving into the question of realizing such a family, the following observation gives simple upper and lower bounds on $a$:

\begin{obs}\label{lem:infinite_family_a_range}
Any ESPP instance of the form $(ak,k,P)$, where $P=[2,p_2,\ldots,p_k]$ must satisfy $2 \leq a < 4$.
\end{obs}
\begin{proof}
To prove the lower bound on $a$, observe that since $P$ is non-decreasing we have $n = ak = \sum_{i=1}^k p_i \geq 2k$.
The upper bound on $a$ follows from the fact that $\slack_1(P) = 2n-1 - s^{n,k} \geq 0$ must hold. 
Therefore, we have $2n > n(n+1)/(2k) > n^2/(2k) = an/2$, implying the required bound.
\end{proof}

\begin{proof}[Proof of Theorem~\ref{thm:inifinite_family}]
Given $a$ in the specified range, we show that there are rational numbers $u,v$ and an integer $d \geq 3$ for which there are infinitely many integers $k$ such that the incomplete ESPP instance $(ak,k,Q_k=[2^{uk}, d^{vk}])$  can be completed to a non-solvable ESPP instance $(ak,k,P)$. By Observation \ref{lem:incomplete_partition_completion} and Proposition \ref{prop:criterion1}, this will be established if the following properties hold: 
\begin{enumerate}
    \item $n=ak$, $\snk=a(ak+1)/2$, $uk$ and $vk$ are integers,\label{integrality_property}
    \item $u,v > 0$ and $u+v<1$,\label{uv_property}
    \item The triple $(ak,k,Q_k=[2^{uk}, d^{vk}])$ is an incomplete ESPP inatance, and\label{extend_property}
    \item The conditions of Proposition \ref{prop:criterion1} hold, with $n=ak$, $e=uk$, and $f=vk$.\label{nonequitability_property}
\end{enumerate} 

From now on, we refer to these as ``Property (1)'',  ``Property (2)'', etc.

The proof proceeds as follows: Lemma~\ref{lem:inifinite_family_extend} defines conditions for properties \eqref{uv_property} and \eqref{extend_property} to hold. Lemma~\ref{lem:inifinite_family_cslack} defines conditions for Property \eqref{nonequitability_property} to hold. 
Lemma~\ref{lem:solution_exists} shows that given $2<a<\frac{24}{7}$ there exist an integer $d\ge3$ and real numbers $u$ and $v$ satisfying the conditions of Lemma~\ref{lem:inifinite_family_extend} and Lemma~\ref{lem:inifinite_family_cslack}. Finally, Lemma~\ref{lem:inifinite_family_inequalities} shows that once we find such values of $d$, $u$, and $v$, then $u$, and $v$ can be chosen to be rational and there are infinitely many values of $k$ such that all four properties $(1)-(4)$ hold. This completes the proof of Theorem~\ref{thm:inifinite_family}.
\end{proof}

In the following lemmas, we derive a set of inequalities that provide sufficient conditions for our requirements. By taking the limit as $k\rightarrow\infty$, we decouple the parameters $(a,u,v,d)$ from the growth of $k$. To guarantee that a solution in this limit space implies feasibility for a finite, sufficiently large $k$, we restrict our search to the interior of the parameter space. That is, we satisfy the limiting conditions via strict inequalities, ensuring that the results are robust to the $o(1)$ corrections present in the pre-limit system.

\begin{lem}\label{lem:inifinite_family_extend}
Given an integer $d \geq 3$, real numbers $2 < a < 4$, $u$ and $v$, and an integer $k$ such that Property \eqref{integrality_property} holds, conditions (\ref{eq:inifinite_family_extend1}) through (\ref{eq:inifinite_family_slack3}) guarantee that properties \eqref{uv_property} and \eqref{extend_property} hold.
\begin{eqnarray}
        v &>& 0                                                                                  \label{eq:inifinite_family_extend1} \\
        v &<& 1-u                                                                                \label{eq:inifinite_family_extend2} \\       
        u &>& \max\left(0,\frac{d-a}{d-2}\right)                                                 \label{eq:inifinite_family_extend3} \\
        u &<& a - \frac{a^2}{4}                                                                  \label{eq:inifinite_family_slack1} \\
        v &<& \frac{2ad - a^2 - 4du + a \sqrt{(2d - a)^2 - 4d(d-2)u}}{2d^2}                      \label{eq:inifinite_family_slack2} \\
        v &>& \frac{2ad - a^2 - 4du - a \sqrt{(2d - a)^2 - 4d(d-2)u}}{2d^2}                      \label{eq:inifinite_family_slack3}
\end{eqnarray}

\end{lem}

\begin{proof}
Clearly, Property~\eqref{uv_property} is guaranteed by \eqref{eq:inifinite_family_extend1}-\eqref{eq:inifinite_family_extend3}.
Since $l=(u+v)k$, we have $l<k$ by \eqref{eq:inifinite_family_extend2}. So, to show that Property~\eqref{extend_property} holds, it remains to show that \eqref{def:inc_part_cond} holds and that $\slack(Q_k)\ge0$. Now, \eqref{def:inc_part_cond} is equivalent to $p_l \leq (n-\sum_{i=1}^l p_i)/(k-l)$, and since $d=p_l$ and $n=ak$, this is equivalent to $d\le \frac{a-2u-dv}{1-u-v}$. The latter is equivalent to $u \ge \frac{d-a}{d-2}$, which holds by \eqref{eq:inifinite_family_extend3}.

To show that $\slack(Q_k) \geq 0$, it suffices to check that $\slack_j(Q_k) \geq 0$ for $j=uk$ and $j=(u+v)k$, by Lemma~\ref{lem:slack_at_end_of_block}. 
The condition for $j=uk$ is:
$$
0 \leq \slack_{uk}(Q_k) = \sum_{i=1}^{2 u k} (n-i+1) - u k \, s^{n,k} = u k (2 n - 2 u k + 1) - u k \frac{n(n+1)}{2k}.
$$
Substituting $n=ak$, dividing by $2uk^2$, and taking the limit, we obtain that the last inequality follows from (\ref{eq:inifinite_family_slack1}).
The condition for $j=(u+v)k$ is:
\begin{eqnarray*}
0  &\leq& \slack_{(u+v)k}(Q_k) = \sum_{i=1}^{(2 u + d v) k} (n-i+1) - (u+v) k s^{n,k} \\
   &=&    (2 u + d v) k \frac{2 n - 2 u k - d v k + 1}{2} - (u+v) k \frac{n(n+1)}{2k},\\
0  &<&    (2 u + d v) (2 a - 2 u - d v) - (u+v) a^2,
\end{eqnarray*}
and we obtain a quadratic inequality for $v$:
\[d^2v^2+(a^2+4du-2ad)v+u(4u+a^2-4a)<0.\]
Solving for $v$ yields that the condition is equivalent to $v_- < v < v_+$ where:
\begin{eqnarray*}
v_\pm &=& \frac{2ad - a^2 - 4du \pm a \sqrt{(2d - a)^2 - 4d(d - 2)u}}{2d^2}.
\end{eqnarray*}
The discriminant is always non-negative by the following, where the first inequality uses  (\ref{eq:inifinite_family_slack1}) to bound $u$ from above:
\begin{eqnarray*}
(2d - a)^2 - 4d(d - 2)u 
&>& (2d - a)^2 - 4d(d - 2)(a-a^2/4) \\
%&=& (2d - a)^2 - ad(d - 2)(4-a) 
&=& (2d - a(d-1))^2 \geq 0.
\end{eqnarray*}
Therefore $v_\pm$ are real solutions, implying that conditions (\ref{eq:inifinite_family_slack2}) 
and (\ref{eq:inifinite_family_slack3}) are sufficient
to guarantee that $\slack_{(u+v)k}(Q_k)$ is non-negative.
\end{proof}

\begin{lem}\label{lem:inifinite_family_cslack}
Given an integer $d \geq 3$, real numbers $2 < a < 4$, $u$ and $v$, and an integer $k$ such that Property ~\eqref{integrality_property} holds, the condition
\begin{equation}
    v > 2a - \frac{a^2}{2} + \max\left(0,\frac{(d - 1)a^2}{d^2} - \frac{2a}{d}\right) - 2u \label{eq:inifinite_family_cslack}
\end{equation}
guarantees that Property~\eqref{nonequitability_property} holds.
\end{lem}

\begin{proof}

We first show that Condition (a) in Proposition~\ref{prop:criterion1} holds, that is, $vk > |C|-2uk$. 
The latter is equivalent to
\begin{equation}\label{ineq:condition_a}
    2 u k + v k > |C| = n - (s^{n,k}-n) + 1 = 2 a k - \frac{a(a k+1)}{2} + 1.
\end{equation}
Dividing by $k$ and taking the limit shows that \eqref{ineq:condition_a} is guaranteed by $v > 2a - a^2/2 - 2u$, which is implied by (\ref{eq:inifinite_family_cslack}).

Now we show that Condition (b) in Proposition~\ref{prop:criterion1} holds. 
Let $l=f-h$, which we now know to be positive. So,
\begin{eqnarray}\label{eq:lb_for_l}
    l = f-h = vk - (|C| - 2uk) = vk + 2uk -(2ak - \frac{a(a k+1)}{2} + 1).
\end{eqnarray}
To show that Condition (b) holds, we need to show $\sum_{i=c-dl}^{c-1}i < l\snk$. 
Changing the summation variable to $j=c-i$ and replacing $c$ by $\snk-n$, we need to show

\begin{eqnarray*}
0 &>& \sum_{j=1}^{d l} (s^{n,k}-n-j) - l s^{n,k} \\
   &=& dl\left(s^{n,k}-n-\frac{dl+1}{2}\right)-ls^{n,k} \\
   &=& d l \left(\frac{a(a k + 1)}{2} - a k - \frac{d l + 1}{2}\right) - l \frac{a(a k+1)}{2}.
\end{eqnarray*}
Dividing by $k l$, applying \eqref{eq:lb_for_l}, and ignoring negligible terms since $k$ is large, yields that Condition (b) is guaranteed by:
\begin{eqnarray*}
0 &>& d \left(\frac{a^2}{2} - a - \frac{d}{2}\left(v + 2u - 2a + \frac{a^2}{2}\right)\right) - \frac{a^2}{2}.
\end{eqnarray*}
Isolating $v$ yields
\[
    v > 2a - \frac{a^2}{2} + \frac{(d - 1)a^2}{d^2} - \frac{2a}{d} - 2u,
\]
which is implied by \eqref{eq:inifinite_family_cslack} as well, completing the proof.
\end{proof}

\begin{lem}\label{lem:solution_exists}
    Given a rational number $2<a<\frac{24}{7}$, there exist an integer $d \geq 3$ and real numbers $u$ and $v$ that satisfy the inequalities  \eqref{eq:inifinite_family_extend1}-\eqref{eq:inifinite_family_slack3} and \eqref{eq:inifinite_family_cslack}.
\end{lem}
\begin{proof}
 Given the, yet unspecified, integer $d \geq 3$, we would like to find $u \in \left(\max(0,\frac{d-a}{d-2}),a-\frac{a^2}{4}\right)$ (as in \eqref{eq:inifinite_family_extend3} and \eqref{eq:inifinite_family_slack1}) such that there is a value of $v$ satisfying inequalities (\ref{eq:inifinite_family_extend1}), (\ref{eq:inifinite_family_extend2}), (\ref{eq:inifinite_family_slack2}), (\ref{eq:inifinite_family_slack3}) and (\ref{eq:inifinite_family_cslack}).
Our first observation is that the $\max$ in  (\ref{eq:inifinite_family_cslack}) is $(d - 1)a^2/d^2 - 2a/d$ if: 
\begin{eqnarray}
a \geq \frac{2d}{d-1} = 2+\frac{2}{d-1}. \label{eq:inifinite_family_redundant_cslack}
\end{eqnarray}
In what follows, we choose $d \geq 3$ to be the smallest possible so that (\ref{eq:inifinite_family_redundant_cslack}) holds:
\begin{eqnarray}
d = 1 + \left\lceil\frac{2}{a-2}\right\rceil.   \label{eq:inifinite_family_choosing_d}
\end{eqnarray}
It follows that:
\begin{eqnarray} 
2+\frac{2}{d-1} \leq a < 2+\frac{2}{d-2}.\label{eq:inifinite_family_a_bounds}
\end{eqnarray}

Therefore, by our choice of $d$ from now on we replace (\ref{eq:inifinite_family_cslack}) by:
\begin{eqnarray}
v > 2a - \frac{a^2}{2} + \frac{(d - 1)a^2}{d^2} - \frac{2a}{d} - 2u.   \label{eq:inifinite_family_cslack1}
\end{eqnarray}

One more consequence of the choice of $d$, proven in Lemma~\ref{lem:infinite_family_redundant_slack3} below, is that inequality (\ref{eq:inifinite_family_slack3}) is redundant in the given system.

We denote by $\mathcal S$ the set of solutions to the system, 
where ${\mathcal S} = {\mathcal S}_1 \cap {\mathcal S}_2$:
\begin{enumerate}
    \item 
    The primary region ${\mathcal S}_1$ is the set of solutions to inequalities
    (\ref{eq:inifinite_family_extend1}), (\ref{eq:inifinite_family_slack1}), (\ref{eq:inifinite_family_slack2}) and (\ref{eq:inifinite_family_cslack1}). 
    In Lemma~\ref{lem:inifinite_family_a_range_primary} we prove that the point $(u_1,v_1) = (a-\frac{a^2}{4},\frac{a^2 (d-1) - 2ad}{d^2})$ is on the boundary of ${\mathcal S}_1$.
    \item 
    The secondary region ${\mathcal S}_2$ is the set of solutions to the rest of the inequalities:
    (\ref{eq:inifinite_family_extend2}) and (\ref{eq:inifinite_family_extend3}).
    In Lemma~\ref{lem:inifinite_family_a_range_secondary} we prove that $(u_1,v_1) \in {\mathcal S}_2$.
\end{enumerate}

The scenario is illustrated in Figure~\ref{fig:inifinite_family_satisfiability_region} for two specific values of $a$.
\begin{figure}[h]
  %\begin{center}
  \centering
  \includegraphics[width=1.0\textwidth]{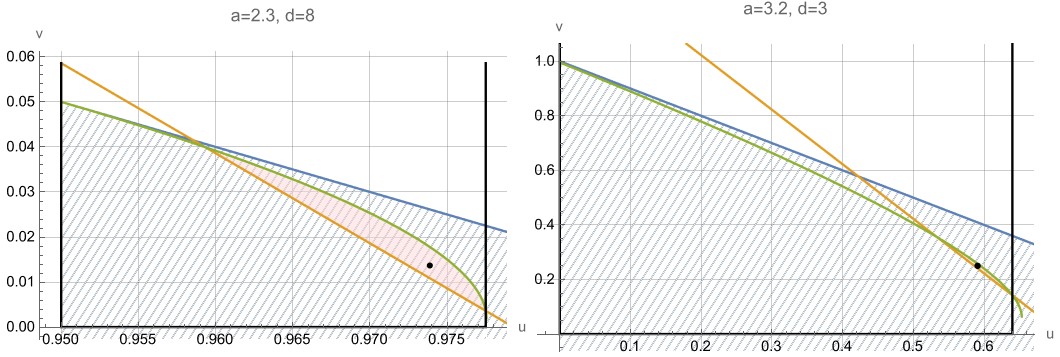}
  %\end{center}
  \caption{
     Satisfiability regions for $a=2.3$ and $a=3.2$. 
     The two vertical black lines and the horizontal black line are respectively (\ref{eq:inifinite_family_extend3}), (\ref{eq:inifinite_family_slack1}) and (\ref{eq:inifinite_family_extend1}). 
     The upper bound (\ref{eq:inifinite_family_slack2}) for $v$  is the green curve, and the lower bound (\ref{eq:inifinite_family_cslack1}) is the orange line. 
     As we show in Lemma~\ref{lem:inifinite_family_a_range_primary}, these two lines intersect at two points $(u_1,v_1)$ and $(u_2,v_2)$ where $u_1$ is on the right vertical line. Thus, the primary region ${\mathcal S}_1$ is the pink region.
     The blue line denotes the upper bound \eqref{eq:inifinite_family_extend2}, and together with the lhs vertical line \eqref{eq:inifinite_family_extend3}, it forms the secondary region ${\mathcal S}_2$, which is colored with hatched lines.
     The black dot is a feasible solution.
     }
  \label{fig:inifinite_family_satisfiability_region}
\end{figure}

Combining these two Lemmas, we conclude that there exists a point $(u,v) $ sufficiently close to $(u_1,v_1)$ so that $(u,v)$ is in the intersection 
$\mathcal S = {\mathcal S}_1 \cap {\mathcal S}_2$, concluding the proof of Lemma~\ref{lem:solution_exists}.
\end{proof}

\begin{lem}\label{lem:infinite_family_redundant_slack3}
    Given a rational number $2<a<\frac{24}{7}$, the integer $d$ given by (\ref{eq:inifinite_family_choosing_d}), and some value of $u$ satisfying $0 < u < a - a^2/4$, 
    the lower bound on $v$ given by (\ref{eq:inifinite_family_slack3}) is non-positive.
    Therefore, inequality (\ref{eq:inifinite_family_slack3}) is made redundant by the $v>0$ constraint (\ref{eq:inifinite_family_extend3}).
\end{lem}

\begin{proof}
   Denoting the rhs of inequality (\ref{eq:inifinite_family_slack3}) by $h(u)$ we have:
    \begin{eqnarray*}
    h(u)    &=&    \frac{2ad - a^2 - 4du - a \sqrt{(2d - a)^2 - 4d(d-2)u}}{2d^2}\\ 
    h'(u)   &=&    \frac{-2 + a(d-2) {\Bigl((2d - a)^2 - 4d(d-2)u\Bigr)}^{-1/2}}{d}\\ 
    h''(u)  &=&    2a(d-2)^2 {\Bigl((2d - a)^2 - 4d(d-2)u\Bigr)}^{-3/2}\\
            &\geq& h''(a-a^2/4) = 2a(d-2)^2 {\Bigl((a(d-1)-2d)^2\Bigr)}^{-3/2}\\
            &=&    2a(d-2)^2 {\Bigl(a(d-1)-2d\Bigr)}^{-3}
             >  2a(d-2)^2 {\biggl(\Bigl(2+\frac{2}{d-2}\Bigr)(d-1)-2d\biggr)}^{-3} \\
            &=& 2a(d-2)^2 {\Bigl(-2+2\frac{d-1}{d-2}\Bigr)}^{-3}> 0.
    \end{eqnarray*}

    Consequently, the maximum of $h(u)$ is obtained in one of the endpoints, $u=0$ or $u=a-a^2/4$.
    Therefore, the Lemma is implied by the following computation:
    \begin{eqnarray*}
    h(0) &=& \frac{2ad - a^2 - a (2d - a)}{2d^2} = 0\\
    h(a-a^2/4) 
    %&=& \frac{2ad - a^2 - 4d(a-a^2/4) - a \sqrt{(2d - a)^2 - 4d(d-2)(a-a^2/4)}}{2d^2}\\ 
    &=& \frac{a^2 (d-1) - 2ad  - a \sqrt{(a(d-1)-2d)^2}}{2d^2}\\  
    &=& \frac{a^2 (d-1) - 2ad  - a \Bigl(a(d-1)-2d\Bigr)}{2d^2} = 0\\  
    \end{eqnarray*}
 
\end{proof}

\begin{lem}\label{lem:inifinite_family_a_range_primary}
Let $2<a<\frac{24}{7}$ be a rational number, and assume the integer $d$ is determined by (\ref{eq:inifinite_family_choosing_d}). Let ${\mathcal S}_1$ be the set of solutions to inequalities (\ref{eq:inifinite_family_extend1}), (\ref{eq:inifinite_family_slack1}), (\ref{eq:inifinite_family_slack2}) and (\ref{eq:inifinite_family_cslack1}). 
Then, the point $(u_1,v_1) = (a-\frac{a^2}{4},\frac{a^2 (d-1) - 2ad}{d^2})$ lies on the boundary of ${\mathcal S}_1$.
\end{lem}

\begin{proof}
Clearly the strict inequality (\ref{eq:inifinite_family_slack1}) does not hold for $u=u_1$ since we have equality instead of the required strict inequality, 
so $(u_1,v_1) \not \in {\mathcal S}_1$.
In order to show that $(u_1,v_1) \in \partial {\mathcal S}_1$, 
we show that 
a) all four constraints are met if we replace $>$ or $<$ by $=$, 
and b) the point $(u_1,v_1)$ has arbitrarily close points $(u,v) \in {\mathcal S}_1$, 

We go over the four inequalities one by one:
\begin{description}
\item[(\ref{eq:inifinite_family_extend1})]    
The claim follows from the fact that by  (\ref{eq:inifinite_family_redundant_cslack}) 
 $v_1 = \frac{a(a (d-1) - 2d)}{d^2} \geq 0$.

\item[(\ref{eq:inifinite_family_slack1})]
Equality is obvious.

\item[(\ref{eq:inifinite_family_slack2})] 
We denote the rhs by $f(u)$ and compute:
\begin{eqnarray*}
    f(u_1) &=& \frac{2ad - a^2 - 4d(a-\frac{a^2}{4}) + a \sqrt{(2d - a)^2 - 4d(d-2)(a-\frac{a^2}{4})}}{2d^2}\\
           % &=& \frac{2ad - a^2 - 4ad + a^2 d + a \sqrt{4d^2 - 4ad + a^2 - 4d^2 a + d^2a^2 + 8da - 2d a^2}}{2d^2}\\
           % &=& \frac{-2ad - a^2 + a^2 d + a \sqrt{4d^2 + 4ad + a^2 - 4d^2 a + d^2 a^2 - 2da^2}}{2d^2}\\
           &=& \frac{-2ad + a^2 (d-1) + a \sqrt{(a (d-1) - 2d)^2}}{2d^2}\\
           &=& \frac{-2ad + a^2 (d-1) + a (a (d-1) - 2d)}{2d^2} = \frac{a^2 (d-1) - 2ad}{d^2} = v_1,
\end{eqnarray*}
where we use (\ref{eq:inifinite_family_redundant_cslack}) for choosing the correct branch of the square root.

\item[(\ref{eq:inifinite_family_cslack1})]
We denote the rhs by $g(u)$ and compute:
\begin{eqnarray*}
    g(u_1) 
    &=& 2a - \frac{a^2}{2} + \frac{(d - 1)a^2}{d^2} - \frac{2a}{d} - 2(a-\frac{a^2}{4})
    = \frac{(d - 1)a^2 - 2 a d}{d^2} = v_1
\end{eqnarray*}
\end{description}
Hence, we have shown a). To show b) we need to show that for a sufficiently small $\epsilon>0$, we have 
\begin{equation}\label{ineq:lower_below_upper}
    f(u_1-\epsilon) > g(u_1-\epsilon).
\end{equation}

This would allow us to choose a value $v$ such that $g(u_1-\epsilon)<v<f(u_1-\epsilon)$. Let $u=u_1-\epsilon$.
Clearly, for $(u,v)$, the inequality
(\ref{eq:inifinite_family_slack1}) is satisfied, 
and since $v_1\ge0$ and the lower bound has a negative slope, the
inequality (\ref{eq:inifinite_family_extend1}) is satisfied as well.

As $g(u)$ is a line of slope $-2$, to show \eqref{ineq:lower_below_upper} we need to show that $f'(u_1) < -2$.
Indeed,
\begin{eqnarray*}
    f'(u_1) 
    &=& \frac{-4 d - 2 a (d-2) d/\sqrt{(2d - a)^2 - 4 (d-2) d u_1}}{2d^2} \\
    &=& \frac{-4 d - 2 a (d-2) d/(a (d-1) - 2d)}{2d^2}
    = -\frac{a (3d - 4) - 4 d}{(a (d-1) - 2 d) d}.
\end{eqnarray*}
Therefore, we need to prove any of the following inequalities:
\begin{eqnarray*}
    a(3d - 4) - 4d &>& 2d(a(d-1)-2d) \\
           4d(d-1) &>& a (2d(d-1)-3d+4) \\
                 a &<& \frac{4d(d-1)}{2d^2 - 5d + 4} = 2 + \frac{2(3d-4)}{2d^2 - 5d + 4},
\end{eqnarray*}
where the last step uses the fact that $2d^2 - 5d + 4 > 0$ for $d \geq 3$.
It is easy to check that for $d=3$ the upper bound on $a$ is 24/7, which holds by the assumption of the Lemma.
For larger values of $d$, by (\ref{eq:inifinite_family_a_bounds}) it is sufficient to show that $\frac{1}{d-2} \leq \frac{3d-4}{2d^2 - 5d + 4}$, 
or equivalently that
\begin{eqnarray*}
0 \leq (3d-4)(d-2) - (2d^2-5d+4) = (d-4)(d-1),
\end{eqnarray*}
which is true for all $d \geq 4$.
\end{proof}

\begin{lem}\label{lem:inifinite_family_a_range_secondary}
Given a rational number $2<a<\frac{24}{7}$ and the integer $d$ given by (\ref{eq:inifinite_family_choosing_d})
then the point $(u_1,v_1) = (a-\frac{a^2}{4},\frac{a^2 (d-1) - 2ad}{d^2})$ is inside the region $\mathcal{S}_2$ defined by the inequalities
(\ref{eq:inifinite_family_extend2}) and (\ref{eq:inifinite_family_extend3}).
\end{lem}

\begin{proof}
    We show that $(u_1,v_1)$ satisfies both inequalities:
    \begin{description}
    
        \item[(\ref{eq:inifinite_family_extend2})]
            \begin{eqnarray*}
                1 - u_1 - v_1 &=& 1 - a + \frac{a^2}{4} - \frac{a^2 (d-1) - 2ad}{d^2}\\ 
                              &=& \frac{4 d^2 - 4 a d^2 + a^2 d^2 - 4 a^2(d-1) + 8ad}{4d^2}\\
                              &=& \frac{(2d - a(d-2))^2}{4d^2} > 0,
            \end{eqnarray*}
            where we used (\ref{eq:inifinite_family_a_bounds}) to obtain the last inequality.
        
        \item[(\ref{eq:inifinite_family_extend3})]
            As $u_1 >0$ obviously holds for all $a<4$, we only need to ascertain that $u_1 > (d-a)/(d-2)$.
            Indeed, as $(d-a)/(d-2) = 1 - (a-2)/(d-2)$ is monotone increasing with $d$, it is sufficient to prove the bound with $d$ replaced by its upper bound $2+2/(a-2)$ from (\ref{eq:inifinite_family_choosing_d}).
            Namely:
            \begin{eqnarray*}
                u_1 - \frac{d-a}{d-2} 
                &\geq&  u_1 - \frac{2+\frac{2}{a-2}-a}{\frac{2}{a-2}}
                = a - \frac{a^2}{4} - \frac{2-(a-2)^2}{2}
                % &=& \frac{4a - a^2 - 4 + 2a^2 + 8 - 8a}{4} 
                = \left(\frac{a-2}{2}\right)^2 > 0.
            \end{eqnarray*}
        
    \end{description}
\end{proof}

\begin{lem}\label{lem:inifinite_family_inequalities}
    Given a rational number $2<a<\frac{24}{7}$, If the system of inequalities \eqref{eq:inifinite_family_extend1}-\eqref{eq:inifinite_family_slack3} and \eqref{eq:inifinite_family_cslack} has a solution consisting of an integer $d>=3$ and real numbers $u_0$ and $v_0$, then there exist rational numbers $u$ and $v$ such that $d$, $u$ and $v$ are also solutions of \eqref{eq:inifinite_family_extend1}-\eqref{eq:inifinite_family_slack3} and \eqref{eq:inifinite_family_cslack}, and such that there are infinitely many values of $k$ for which $ak$, $a(ak+1)/2$, $uk$ and $vk$ are integers.
\end{lem}

\begin{proof}
Let $a=m/r$ with $\gcd(m,r)=1$. 
The strict inequalities imply that any $(u,v)$ sufficiently close to $(u_0,v_0)$ is a solution as well, and in particular, one may find a solution with rational $u,v$ whose denominator is a large prime $p > 2r^2$. The result follows from the next claim.

\begin{claim}\label{lem:integrality_condition_a}
There is a positive integer $k_0$, depending only on $a$, so that for all $k = k_0 + 2 r^2 i$, $i \in \NZ$, both $n=ak$ and $s^{n,k}=n(n+1)/(2k)$ are integers.
\end{claim}

\begin{claimproof}
Let $S$ be the set of integers satisfying the integrality requirements for $n$ and $s^{n,k}$. So, $S=S_1\cap S_2$, where  
$n=ak$ is integral for:
\begin{eqnarray*}
S_1 &=& \{k \in \N : m k \equiv 0 \Mod r \} = r \N,
\end{eqnarray*}
and $s^{n,k}$ is integral for:
\begin{eqnarray*}
S_2 &=& \{k \in \N : s^{a k,k} \in \N \} = \{k \in \N : \frac{a(a k+1)}{2} \in \N \}\\
  &=& \{k \in \N : \frac{m^2 k + m r}{2r^2} \in \N \} = \{k \in \N : m^2 k \equiv - m r \pmod {2r^2} \}.
\end{eqnarray*}

To continue, we distinguish the case of odd and even $m$.
For the odd $m$ case, $m$ and $2r^2$ are relatively prime, so $m$ has an inverse modulo $2r^2$. 
Then: 
\begin{eqnarray*}
S_2^{\textrm{odd}} = \{k \in \N : m k \equiv - r \Mod {2r^2} \} = (-m^{-1}r \mod 2r^2) + 2r^2 \NZ.
\end{eqnarray*}
For the even $m=2t$ case, it follows that $r$ is odd and both $2$ and $t$ have an inverse modulo $r^2$ and we get:
\begin{eqnarray*}
S_2^{\textrm{even}} 
&=& \{k \in \N : 4 t^2 k \equiv -2 t r \Mod {2r^2} \}
= \{k \in \N : 2 t^2 k \equiv -t r \Mod {r^2} \}\\
&=& \{k \in \N : m k \equiv -r \Mod {r^2} \}
= (-m^{-1} r \textrm{ mod } r^2) + r^2 \NZ.
\end{eqnarray*}
In both cases $r$ divides the additive offset, so $S = S_1 \cap S_2 = S_2 \supseteq S_2^{\textrm{odd}}\cap S_2^{\textrm{even}}=S_2^{\textrm{odd}}$, yielding the required result with $k_0=-m^{-1} r\mod 2r^2$.
\end{claimproof}

Since the denominator $p$ of $u$ and $v$ is prime to $2r^2$, there are infinitely many elements in the sequence $\{k_0 + 2 r^2 i\}_{i \in \NZ}$ that are equal to 0 $\pmod p$. This completes the proof of Lemma~\ref{lem:inifinite_family_inequalities}.

\end{proof}

The next proposition indicates that we cannot use the criterion in Proposition~\ref{prop:criterion1} to extend the result of Theorem~\ref{thm:inifinite_family} to values of $a$ such that $\frac{24}{7}\le a<4$.

\begin{prop}\label{lem:infinite_family_a_too_large}
%For any integer $d \geq 3$ and $a > \frac{4d(d-1)}{4-5d+2d^2}$ the system of 
For any number $a \geq \frac{24}{7}$ and any integer $d \geq 3$,
the system of inequalities \eqref{eq:inifinite_family_extend1}-\eqref{eq:inifinite_family_slack3} and \eqref{eq:inifinite_family_cslack} admits no solution.
\end{prop}
\begin{proof}
We claim that the system consisting of the three inequalities 
(\ref{eq:inifinite_family_slack1}), 
(\ref{eq:inifinite_family_slack2}) and 
(\ref{eq:inifinite_family_cslack1}) has no solution.
As (\ref{eq:inifinite_family_cslack1}) is a weaker constraint than \eqref{eq:inifinite_family_cslack}, it follows that the entire system has no solution. 
Indeed, as in the proof of Lemma~\ref{lem:inifinite_family_a_range_primary}, 
let us denote 
the rhs of (\ref{eq:inifinite_family_slack2}) by $f(u)$ and 
the rhs of (\ref{eq:inifinite_family_cslack1}) by $g(u)$.
We solve the equality $g(u) = f(u)$, 
where we denote $\Delta=(2d - a)^2 - 4d(d-2)u$:
\begin{eqnarray*}
2a - \frac{a^2}{2} + \frac{(d - 1)a^2}{d^2} - \frac{2a}{d} - 2u
&=& \frac{2ad - a^2 - 4du + a \sqrt{\Delta}}{2d^2}\\
4 a d^2 -a^2  d^2 + 2(d - 1)a^2 - 4ad - 4 d^2 u
&=& 2ad - a^2 - 4du + a \sqrt{\Delta}
\end{eqnarray*}
\begin{eqnarray}
4 a d^2 - a^2 d^2 + 2 a^2 d  - a^2 - 6ad - 4 d(d-1) u
&=& a \sqrt{\Delta} \label{eq:tmp1}
\end{eqnarray}

The next step is to square both sides of the equation. However, we first check that the lhs is non-negative. 
As the coefficient of $u$ is negative, 
we may substitute the upper bound (\ref{eq:inifinite_family_slack1}), $a-a^2/4$, for $u$:
\begin{eqnarray*}
\textrm{lhs}
&>& 4 a d^2 - a^2 d^2 + 2 a^2 d - a^2 - 6ad - (d-1)d(4a - a^2)\\
&=& 4 a d^2 - a^2 d^2 + 2 a^2 d  - a^2 - 6ad - 4a d^2 + 4ad + a^2 d^2 - a^2 d\\
&=& a^2 d - a^2 - 2ad = a ((d-1) a - 2d) \geq 0,\\
\end{eqnarray*}
where the last inequality is true for $a \geq 24/7 > 3$ and $d \geq 3$.

We continue by taking the square of (\ref{eq:tmp1}):
\begin{eqnarray*}
0 &=& \left(4 a d^2 - a^2 d^2 + 2 a^2 d - a^2 - 6ad - 4 d(d-1) u\right)^2\\
  &&- a^2 \left((2d - a)^2 - 4d(d-2)u\right) = A u^2 + B u + C, \\
\textrm{where:}&&\\
A &=& 16 (d-1)^2 d^2\\
B &=& 4 a d \left(a (2d^3 - 6d^2 + 7d - 4) - 4d (2d^2-5d+3)\right)\\
C &=& (a-4) a^2 (d-2) d \left(a (d^2 - 2d + 2) - 4 (d-1) d\right)\\
B^2-4AC &=& 16 a^2 d^2 \left(a (2d^2 - 5d + 4) - 4 (d-1) d\right)^2 \geq 0.
\end{eqnarray*}
Solving the quadratic equation for $u$ yields the following two solutions:
\begin{eqnarray*}
u_1 &=& a-\frac{a^2}{4}\\
u_2 &=& \frac{a (d - 2) (4 (d - 1) d - a (d^2 - 2 d + 2))}{4 (d - 1)^2 d}.
\end{eqnarray*}
The range of $u$ for which $f(u) > g(u)$ is the open interval between $u_1$ and $u_2$. 

Because of requirement (\ref{eq:inifinite_family_slack1}) that $u < u_1$, 
it follows that a necessary condition for the inequalities to have a solution is:
\begin{eqnarray*}
0 &<& u_1 - u_2 = \frac{a }{4 (d - 1)^2 d}\left[4 (d-1) d - a (2 d^2 - 5d + 4)\right]\\
a &<& h(d) = \frac{4 (d-1) d}{2 d^2 - 5d + 4}.
\end{eqnarray*}

It can be verified that $h(3) = 24/7$, implying that there is no solution with $d=3$.
Furthermore, the function $h$ is monotone decreasing for $d \geq 3$ since:
\begin{eqnarray*}
h'(d)=-\frac{4 (4 - 8 d + 3 d^2)}{(4 - 5 d + 2 d^2)^2},
\end{eqnarray*}
which is negative for $d \geq 3$. 
Therefore, there is no value of $a$ for which $24/7 \leq a < h(d)$ for any $d \geq 3$.
\end{proof}

\begin{rem}
    This does not mean that non-solvable ESPP instances $(ak,k,P)$, where $a\ge\frac{24}{7}$, do not exist. Only that we cannot use the criterion in Proposition~\ref{prop:criterion1} to find them. So far, we have not found such instances in our searches.
\end{rem}

%%%%%%%%%%%%%%%%%%%%%%%%%%%%%%%%%%%%%%%%%%%%%%%%%%%%%%%%%%%%%%%%%%%%%%%%%%%%%%%%%%%%%%%%%%%%%%%%%%%%%%%%%%%

\section{A fractional version of the equal sum partition problem}\label{sec:fractional}

In this section, we explore a fractional version of the equal sum partition problem, and obtain that in this case being solvable and satisfying the slack condition are equivalent. This problem is a special case of a family of classic linear programming problems, known as \emph{pooling problems}, \emph{blending problems}, or \emph{fluid mixing problems} (as we shall refer to them here). 
Introductions to these problems and their different solutions can be found in many textbooks on linear programming and operations research, two of which are  \cite{marshall} and \cite{winston}. 
One possible physical interpretation of a fluid mixing problem is as follows: 
Given $n$ containers full of different liquids, where container $i$ has volume $u_i>0$ and mass $a_i$,
is it possible to pour the liquids into $k$ empty containers, where container $j$ has volume $v_j>0$ so that its mass will be $b_j$?
Formally:
\begin{defn}
    The 4-tuple $\pi=(\overline{a}, \overline{u}, \overline{b}, \overline{v})$ is an instance of the Fluid Mixing Problem with $n$ source containers and $k$ target containers if 
    $\overline{a} \in \R_{\geq 0}^n$, $\overline{u} \in \R_{> 0}^n$,  
    $\overline{b} \in \R_{\geq 0}^k$, $\overline{v} \in \R_{> 0}^k$
    and the following conservation equalities hold:
    \begin{equation}\label{eq:conservation_Of_volume_and_mass}
        \sum_{i=1}^n a_i = \sum_{j=1}^k b_j\qquad\mbox{ and }\qquad
        \sum_{i=1}^n u_i = \sum_{j=1}^k v_j.
    \end{equation}
    Furthermore, we assume that the source and target containers are ordered in descending order by density, that is,
    $a_1/u_1 \geq \cdots \geq a_n/u_n$ and $b_1/v_1 \geq \cdots \geq b_k/v_k$.     
\end{defn}
    
\begin{rem}
    For technical reasons we allow $u_1 = a_1 = 0$ and/or $v_1 = b_1 = 0$.
\end{rem}
  
\begin{defn}\label{def:fluid_mixing}
    We say that an instance $\pi=(\overline{a}, \overline{u}, \overline{b}, \overline{v})$ of the Fluid Mixing Problem is \emph{solvable}, if there exist a solution $X = [x_{i,j}]_{i \in [n],j \in [k]}$ such that:
    \begin{align}
        x_{i,j}                   &\ge 0     & \textrm{for all }i \in [n] \textrm{ and }j \in [k],\\ 
        \sum_{j=1}^k x_{i,j}      &= 1       & \textrm{for all }i \in [n] \label{eq:frac_primal1},\\
        \sum_{i=1}^n u_i x_{i,j}  &= v_j     & \textrm{for all }j \in [k] \label{eq:frac_primal2},\\
        \sum_{i=1}^n a_i x_{i,j}  &= b_j     & \textrm{for all }j \in [k] \label{eq:frac_primal3}.
    \end{align}
\end{defn}
The solution matrix $X = [x_{i,j}]_{i \in [n],j \in [k]}$ is sometimes referred to in the literature as a \emph{transfer plan}.

\noindent
Adopting the physical interpretation of the problem described above, each of the variables $x_{i,j}$ represents the portion of container $i$ that is poured into container $j$.

The fluid mixing problem includes the fractional relaxation of the equal sum partition problem~\ref{prob:nk_partition}. 
Indeed, given an ESPP instance $(n,k,P=[p_1,\ldots,p_k])$, we set $a_i = n-i+1$, $u_i=1$, $b_j=\snk$ and $v_j=p_j$.
If one also adds the integrality constraint $x_{i,j} \in \{0,1\}$, this is exactly the equal sum partition problem.
Furthermore, the following definition of slack generalizes the previous definition to the more general problem.
For the equal sum partition problem, the definition coincides with Definition~\ref{def:slack} of $\slack(P)$.

\begin{defn}[Fractional Slack]\label{def:frac_slack}
    Let $\pi=(\overline{a}, \overline{u}, \overline{b}, \overline{v})$ be an instance of the Fluid Mixing Problem. Define $U_0=V_0=0$, and for $1\le i\le n$ and $1\le j \le k$
    Define $U_i = \sum_{i'=1}^i u_{i'}$ and  $V_j = \sum_{j'=1}^j v_{j'}$.
    Then: 
    \begin{align*}
        \slack(\pi)    &= \min_{1 \leq l \leq k-1} \slack_l(\pi),    \;\;\;\textrm{where}\\
        \slack_l(\pi)  &= \sum_{i=1}^n \frac{a_i}{u_i} \cdot \mu_{i,l} - \sum_{j=1}^l b_j,  \;\;\;\textrm{where}\\
        \mu_{i,l} &= \max(0,\min(u_i,V_l-U_{i-1})).
    \end{align*}
\end{defn}
\noindent
Furthermore, if $u_i=0$, then necessarily $\mu_{i,l}=0$ and the summand $\frac{a_i}{u_i} \cdot \mu_{i,l}$ in the expression for $\slack_l(\pi)$ is regarded as zero.

We shall refer to the condition $\slack(\pi) \geq 0$ as {\bf satisfying the slack condition}. It is equivalent to the condition of \emph{majorization} appearing in the literature (see \cite{marshall}). 
It is a known classical result that majorization (or satisfiability of the slack condition) guarantees that the fluid mixing problem is solvable via a greedy algorithm. 
However, for the sake of the next section, we need an algorithm (Algorithm~\ref{alg:fluid_mixing_solution_rec}) that guaranties the property detailed in Lemma~\ref{lem:non_vanishing_last_block}. 
Theorems~\ref{thm:fluid_mixing_implies_slack} and~\ref{thm:slack_implies_fluid_mixing} correspond to the two directions of Theorem~\ref{thm:fractional}. Moreover, the proof of Theorem~\ref{thm:slack_implies_fluid_mixing} establishes the correctness of Algorithm~\ref{alg:fluid_mixing_solution_rec}. 

\begin{thm}\label{thm:fluid_mixing_implies_slack}
    A solvable instance $\pi$ of the fluid mixing problem satisfies the slack condition.
\end{thm}

\begin{proof}
    Given an instance $\pi=(\overline{a}, \overline{u}, \overline{b}, \overline{v})$ of the fluid mixing problem with the solution $x_{i,j}$,
    we need to prove that $\slack_l(\pi) \geq 0$ for all $l=1,\ldots,k$.
    Given $l$, we define $M_{i,l} = \sum_{i'=1}^i \mu_{i',l}$.
    So, if $U_i \leq V_l$ then $\mu_{i',l} = u_{i'}$ for all $i' \le i$ and therefore $M_{i,l} = U_i$.
    Otherwise, $U_i > V_l$ and let $i_0$ be the maximal $i$ where $U_i \leq V_l$. 
    Then $M_{i,l} = \sum_{i'=1}^{i_0-1} u_{i'} + V_l - U_{i_0-1} = V_l$.
    Therefore:
    \begin{equation}\label{eq:sum_of_mu}
        M_{i,l}  = \sum_{i'=1}^i \mu_{i',l} = \min(U_i,V_l).
    \end{equation}
    Given the solution $x_{i,j}$ of $\pi$, we denote $\alpha_{i,l} = u_i \sum_{j=1}^l x_{i,j}$ and define  
    \begin{eqnarray*}
        A_{i,l} &=& \sum_{i'=1}^i \alpha_{i',l}  
             = \sum_{i'=1}^i u_{i'} \sum_{j=1}^l x_{i',j} 
             \leq \min\!\!\left( \sum_{i'=1}^i u_{i'} \sum_{j=1}^k x_{i',j}, \,
                             \sum_{j=1}^l \sum_{i'=1}^n u_{i'} x_{i',j}\right).
    \end{eqnarray*}
    Using \eqref{eq:frac_primal1} for the first argument of the minimum and \eqref{eq:frac_primal2} for the second,
    it follows that:
    \begin{eqnarray*}
             A_{i,l} &\leq&  \min\!\!\left( \sum_{i'=1}^i u_{i'}, \sum_{j=1}^l v_j\right) = \min(U_i,V_l) = M_i.
    \end{eqnarray*}
    Then, applying \eqref{eq:frac_primal3} to expand $b_j$, we get the following:
    \begin{eqnarray*}
        \slack_l(\pi)  &=& \sum_{i=1}^n \frac{a_i}{u_i} \cdot \mu_{i,l} - \sum_{j=1}^l b_j 
                        = \sum_{i=1}^n \frac{a_i}{u_i} \cdot \mu_{i,l} - \sum_{j=1}^l \sum_{i=1}^n a_i x_{i,j} \\
                        &=& \sum_{i=1}^n \frac{a_i}{u_i} \cdot \mu_{i,l} - \sum_{i=1}^n \frac{a_i}{u_i} \left(u_i \sum_{j=1}^l x_{i,j}\right) 
                         = \sum_{i=1}^n \frac{a_i}{u_i} \cdot ( \mu_{i,l} - \alpha_{i,l} ).
    \end{eqnarray*}
    The next step is to define $\Delta_i = \frac{a_i}{u_i} - \frac{a_{i+1}}{u_{i+1}}$ for $i<n$ and $\Delta_n = \frac{a_n}{u_n}$. 
    Note that $\Delta_i \geq 0$ for all $i$ by our assumptions.
    Then:
    \begin{eqnarray*}
        \slack_l(\pi) 
          &=& \sum_{i=1}^n \frac{a_i}{u_i} \cdot ( \mu_{i,l} - \alpha_{i,l} )
           =  \sum_{i=1}^n \left( \sum_{i'=i}^n \Delta_{i'} \right) \cdot ( \mu_{i,l} - \alpha_{i,l} )  \\
          &=&  \sum_{i'=1}^n \Delta_{i'} \sum_{i=1}^{i'} ( \mu_{i,l} - \alpha_{i,l} )
           =  \sum_{i'=i}^n \Delta_{i'} ( M_{i',l} - A_{i',l} ) \geq 0,
    \end{eqnarray*}
    as claimed.    
\end{proof}

\begin{thm}\label{thm:slack_implies_fluid_mixing}
    An instance $\pi$ of the fluid mixing problem satisfying the slack condition is solvable by means of \Cref{alg:fluid_mixing_solution_rec}.
\end{thm}

\begin{algorithm}[H]
\caption{Recursive Solution For The Fluid Mixing Problem}\label{alg:fluid_mixing_solution_rec}
\begin{algorithmic}[1]
\Function{FluidMixingSolve}{$\pi=(\overline{a}, \overline{u}, \overline{b}, \overline{v})$}
\Comment{Assumes that $\slack(\pi) \ge 0$}
    \State $n \gets $ number of source containers for $\pi$
    \State $k \gets $ number of target containers for $\pi$
    \If{$k=1$}
        \State \textbf{return } the $n \times 1$ all one matrix $X \gets [1_{n \times 1}]$
    \ElsIf{$u_1=0$} \Comment{Null source container, Lemma~\ref{lem:null_u}}
        \State $\overline{a'} \gets (a_2,\ldots,a_n)$, $\overline{u'} \gets (u_2,\ldots,u_n)$
        \State $\pi' \gets (\overline{a'}, \overline{u'}, \overline{b}, \overline{v})$ \label{algline:fluid_mixing_solution_rec:null_u1}
        \State $X' \gets FluidMixingSolve(\pi')$
        \State \textbf{return } the $n \times k$ matrix $X \gets \begin{bmatrix}(\delta_1)_{1 \times k} \\[4pt]  X'\end{bmatrix}$                
                                                                                   \label{algline:fluid_mixing_solution_rec:null_u2}
    \ElsIf{$v_1=0$} \Comment{Null target container, Lemma~\ref{lem:null_v}}
        \State $\overline{b'} \gets (b_2,\ldots,b_k)$, $\overline{v'} \gets (v_2,\ldots,v_k)$
        \State $\pi' \gets (\overline{a}, \overline{u}, \overline{b'}, \overline{v'})$ \label{algline:fluid_mixing_solution_rec:null_v1}
        \State $X' \gets FluidMixingSolve(\pi')$
        \State \textbf{return } the $n \times k$ matrix $X \gets \begin{bmatrix}0_{n \times 1} & X' \end{bmatrix}$    \label{algline:fluid_mixing_solution_rec:null_v2}
    \ElsIf{$b_1/v_1=b_2/v_2$} \Comment{Merge two target containers, Lemma~\ref{lem:merge_v}}
        \State $\overline{b'} \gets (b_1+b_2,b_3\ldots,b_k)$, $\overline{v'} \gets (v_1+v_2,v_3\ldots,v_k)$
        \State $\gamma \gets v_1/(v_1+v_2)$                                              \label{algline:fluid_mixing_solution_rec:merge_gamma}
        \State $\pi' \gets (\overline{a}, \overline{u}, \overline{b'}, \overline{v'})$   \label{algline:fluid_mixing_solution_rec:merge_v1}
        \State $X' \gets FluidMixingSolve(\pi')$
        \State \textbf{return } the $n \times k$ matrix 
               $X \gets \begin{bmatrix}\gamma X'_{1:n,1} & (1-\gamma)X'_{1:n,1} & X'_{1:n,2:k}\end{bmatrix}$ 
        \label{algline:fluid_mixing_solution_rec:merge_v2}
    \Else \Comment{Pour from first source container to first target container, Lemma~\ref{lem:pour_from_first}}
        \State $\lambda \gets \min\!\left(1,\frac{v_1}{u_1},\frac{b_1 v_2 \,-\, v_1 b_2}{a_1 v_2 \,-\, u_1 b_2}\right)$ \label{algline:fluid_mixing_solution_rec:pour_lam}
        \State $\overline{a'} \gets (a_1 - \lambda a_1,a_2,\ldots,a_n)$, $\overline{u'} \gets (u_1 - \lambda u_1,u_2,\ldots,u_n)$
        \State $\overline{b'} \gets (b_1 - \lambda a_1,b_2,\ldots,b_k)$, $\overline{v'} \gets (v_1 - \lambda u_1,v_2,\ldots,v_k)$
        \State $\pi' \gets (\overline{a'}, \overline{u'}, \overline{b'}, \overline{v'})$ \label{algline:fluid_mixing_solution_rec:pour1}
        \State $X' \gets FluidMixingSolve(\pi')$
        \State \textbf{return } the $n \times k$ matrix 
            $X \gets \begin{bmatrix} \lambda(\delta_1)_{1 \times k} + (1-\lambda)X'_{1,1:k}\\[4pt] X'_{2:n,1:k} \end{bmatrix}$ 
            \label{algline:fluid_mixing_solution_rec:pour2}
    \EndIf
\EndFunction
\end{algorithmic}
\end{algorithm}

\noindent
For ease of implementation, we provide an iterative version of \Cref{alg:fluid_mixing_solution_rec} in Appendix~\ref{app:iterative}.

\begin{proof}[Proof of \Cref{thm:slack_implies_fluid_mixing}]
    Given an instance $\pi=(\overline{a}, \overline{u}, \overline{b}, \overline{v})$ of the fluid mixing problem satisfying the slack condition, 
    we prove that \Cref{alg:fluid_mixing_solution_rec} solves the problem by exhibiting the solution $X$.
    The algorithm consists of a sequence of reduction steps, each applying one of the four lemmas~\ref{lem:null_u}--~\ref{lem:pour_from_first},  
    until it is eventually reduced to the base case.
    
    The proof proceeds by induction on the number of reduction steps $m$, which are needed to reduce to the base case.
    The base case $m=0$ consists of all problems with $k=1$, in which case the solution is $X = 1_{n \times 1}$. It corresponds to pouring the content of all the source containers into the single target container.

    Assume by the induction hypothesis that all fluid mixing problems requiring up to $m$ reduction steps, satisfying the slack condition, are solved by the algorithm. Let $\pi$ be a fluid mixing problem, satisfying the slack condition, requiring $m+1$ reduction steps. We argue that \Cref{alg:fluid_mixing_solution_rec} solves the problem. We analyze the algorithm by considering the following four step-types, ordered in descending priority, handled by the four lemmas~\ref{lem:null_u}--~\ref{lem:pour_from_first}, respectively. 
    For each case, the corresponding step reduces the problem $\pi$ to a problem $\pi'$ that requires only $m$ steps to solve.
    \begin{enumerate}
    \item 
        If the problem $\pi$ has a null source container, $u_1=0$, then by Lemma~\ref{lem:null_u}, the problem $\pi'$ defined on line \ref{algline:fluid_mixing_solution_rec:null_u1} satisfies the slack condition, and given its solution $X'$, line \ref{algline:fluid_mixing_solution_rec:null_u2} provides a solution $X$ for the problem $\pi$.   
    \item 
        If the problem $\pi$ has a null target container, $v_1=0$, then by Lemma~\ref{lem:null_v}, the problem $\pi'$ defined on line \ref{algline:fluid_mixing_solution_rec:null_v1} satisfies the slack condition, and given its solution $X'$, line \ref{algline:fluid_mixing_solution_rec:null_v2} provides a solution $X$ for the problem $\pi$.   
    \item 
        If the first two target containers of $\pi$ have the same density, $b_1/v_1=b_2/v_2$, then by Lemma~\ref{lem:merge_v}, the problem $\pi'$ defined on line \ref{algline:fluid_mixing_solution_rec:merge_v1} satisfies the slack condition, and given its solution $X'$, line \ref{algline:fluid_mixing_solution_rec:merge_v2} provides a solution $X$ for the problem $\pi$.   
    \item 
        Otherwise, we pour a portion $\lambda$ from the first source container to the first target container until one of the following stopping criteria, whichever happens first, is met: (i) the first source container is exhausted, (ii) the first target container reaches capacity, or (iii) the density of the first target equals that of the second. The value of $\lambda$ computed by line \ref{algline:fluid_mixing_solution_rec:pour_lam} satisfies the requirements of Lemma~\ref{lem:pour_from_first} and therefore the problem $\pi'$ defined on line \ref{algline:fluid_mixing_solution_rec:pour1} satisfies the slack condition. Given its solution $X'$, line \ref{algline:fluid_mixing_solution_rec:pour2} provides a solution $X$ for the problem $\pi$.   
    \end{enumerate}

    It remains to show that for any fluid mixing problem $\pi$, the number of required reduction steps $m$ is finite.
    This follows from the fact that in cases (1) through (3) the value of $n+k$ for the reduced problem $\pi'$ is strictly smaller than the corresponding value for $\pi$.
    Moreover, in case (4), the reduced problem $\pi'$ satisfies at least one of the conditions for cases (1), (2), or (3), as shown by the following claim.

    \begin{claim}
        In case (4), the reduced problem $\pi'$ satisfies at least one of the conditions for cases (1), (2), or (3).
    \end{claim}
    \begin{claimproof}
        Note that, since we are in case (4), the values of $v_1$ and $u_1$ are strictly positive and $b_2/v_2 > b_1/v_1$.
        We distinguish three cases, according to the value of $\lambda$ computed by the algorithm.
        If $\lambda=1$, the entire source container is poured, so the resulting problem $\pi'$ has a null source container, which is case (1).
        Otherwise, if $\lambda=v_1/u_1$, the entire target container is filled and $\pi'$ has a null target container, which is case (2).
        Otherwise, $\lambda = \frac{b_1 v_2 \,-\, v_1 b_2}{a_1 v_2 \,-\, u_1 b_2}$, which is positive by Lemma~\ref{lem:pour_from_first}.
        Computing the new density of the first target container shows that
        \begin{equation*}
            \frac{b'_1}{v'_1} 
            = \frac{b_1 - \lambda a_1}{v_1 - \lambda u_1} 
            = \frac{b_1(a_1 v_2 \,-\, u_1 b_2) - a_1 (b_1 v_2 \,-\, v_1 b_2)}{v_1(a_1 v_2 \,-\, u_1 b_2) - u_1 (b_1 v_2 \,-\, v_1 b_2)}
            = \frac{b_2 (a_1 v_1 - b_1 u_1)}{v_2 (v_1 a_1 - u_1 b_1)}
            = \frac{b_2}{v_2},
        \end{equation*}
        implying that the resulting problem $\pi'$ is case (3).
    \end{claimproof}
\end{proof}

\begin{lem}\label{lem:null_u}
    Consider an instance $\pi=(\overline{a}, \overline{u}, \overline{b}, \overline{v})$ of the fluid mixing problem satisfying the slack condition, 
    where $u_1=a_1=0$.
    Let $\pi'=(\overline{a'}, \overline{u'}, \overline{b}, \overline{v})$ be the problem obtained by deleting the null source container,
    with $\overline{a'} = (a_2,\ldots,a_n)$ and $\overline{u'} = (u_2,\ldots,u_n)$.
    Then $\pi'$ is a valid problem with $\slack_l(\pi')=\slack_l(\pi)$ for $l \in [k]$, 
    and if $\pi'$ is solvable so is $\pi$.
    Given a solution $X'$ for $\pi'$, then $X = \begin{bmatrix}(\delta_1)_{1 \times k} \\[4pt]  X'\end{bmatrix}$ is a solution for $\pi$.
\end{lem}

\begin{proof}
    Given $\pi$ with $u_1=a_1=0$, define $\pi'$ as in the lemma.
    Then $\pi'$ is a valid fluid mixing problem, as it satisfies the conservation equality \eqref{eq:conservation_Of_volume_and_mass}.
    If $\pi'$ is solvable with solution $x'_{i,j}$, we define $x_{i,j}$ as
    \[
        x_{i,j} = \begin{cases}
           x'_{i-1,j} & \textrm{ if } i \geq 2 \\
           0          & \textrm{ if } i=1, j>1 \\
           1          & \textrm{ if } i=j=1 
        \end{cases}    
    \]
    To see that $X=[x_{i,j}]_{i \in [n],j \in [k]}$ is a solution for $\pi$, note that \eqref{eq:frac_primal1} is satisfied, as $\sum_{j=1}^k x_{i,j} = \sum_{j=1}^k x'_{i,j} = 1$ for $i \geq 2$, and as $\sum_{j=1}^k x_{1,j} = 1$,
    and furthermore, equations \eqref{eq:frac_primal2} and \eqref{eq:frac_primal3} hold because $u_1=a_1=0$ so the value of $x_{1,j}$ does not affect the sums.
    As for the slack, observe that $\mu_i = \mu_{i-1}$ for $i \geq 2$ and that $\mu_1 = 0$.
    Therefore, $\slack_l(\pi) = \slack_l(\pi')$ for all $l$ as claimed.
\end{proof}

\begin{lem}\label{lem:null_v}
    Consider an instance $\pi=(\overline{a}, \overline{u}, \overline{b}, \overline{v})$ of the fluid mixing problem satisfying the slack condition, 
    where $v_1=b_1=0$.
    Let $\pi'=(\overline{a}, \overline{u}, \overline{b'}, \overline{v'})$ be the problem obtained by deleting the null target container,
    with $\overline{b'} = (b_2,\ldots,b_n)$ and $\overline{v'} = (v_2,\ldots,v_n)$.
    Then $\pi'$ is a valid problem with $\slack_l(\pi')=\slack_{l+1}(\pi)$ for $l \in [k-1]$, 
    and if $\pi'$ is solvable so is $\pi$. 
    Given a solution $X'$ for $\pi'$, then $X = \begin{bmatrix}0_{n \times 1} & X' \end{bmatrix}$ is a solution for $\pi$.
\end{lem}

\begin{proof}
    Given $\pi$ with $v_1=b_1=0$, define $\pi'$ as in the lemma.
    Then $\pi'$ is a valid fluid mixing problem as it satisfies the conservation equality \eqref{eq:conservation_Of_volume_and_mass}.
    If $\pi'$ is solvable with solution $x'_{i,j}$, we define $x_{i,j}$ as $x'_{i,j-1}$ if $j \geq 2$ and as zero if $j=1$.
    Then, \eqref{eq:frac_primal1} is satisfied as an added zero does not change a sum. 
    Equations \eqref{eq:frac_primal2} are satisfied for $j \geq 2$ as $\sum_{i=1}^n u_i x_{i,j} = \sum_{i=1}^n u_i x'_{i,j-1} = v'_{j-1} = v_j$
    and as the sum is zero for $j=1$. Similar computation holds for \eqref{eq:frac_primal3}, implying that $\pi$ is solvable.
    As for the slack, $\mu'_{i,l-1} = \mu_{i,l}$ for $l \geq 2$ and $\mu_{i,1}=0$ for all $i$.
    Therefore, $\slack_l(\pi) = \slack_{l-1}(\pi')$ for $l \geq 2$ as claimed.
\end{proof}

\begin{lem}\label{lem:merge_v}
    Consider an instance $\pi=(\overline{a}, \overline{u}, \overline{b}, \overline{v})$ of the fluid mixing problem satisfying the slack condition, 
    with $k \geq 2$ and with strictly positive vectors $\overline{u}, \overline{v}$, where $b_1/v_1=b_2/v_2$.
    Let $\pi'=(\overline{a}, \overline{u}, \overline{b'}, \overline{v'})$ be the problem obtained by merging the two containers,  
    with $\overline{b'} = (b_1+b_2,b_3,\ldots,b_k)$ and $\overline{v'} = (v_1+v_2,v_3,\ldots,v_k)$.
    Then $\pi'$ is a valid problem with $\slack_l(\pi')=\slack_{l+1}(\pi)$ for $l \in [k-1]$, 
    and if $\pi'$ is solvable so is $\pi$.
    Given a solution $X'$ for $\pi'$, then $X = \begin{bmatrix}\gamma X'_{1:n,1} & (1-\gamma)X'_{1:n,1} & X'_{1:n,2:k}\end{bmatrix}$ is a solution for $\pi$, where $\gamma = v_1/(v_1+v_2)$.
\end{lem}

\begin{proof}
    Given $\pi$ with $b_1/v_1=b_2/v_2$, define $\pi'$ as in the lemma.
    Then $\pi'$ is a valid fluid mixing problem as it satisfies the conservation equality \eqref{eq:conservation_Of_volume_and_mass}.
    If $\pi'$ is solvable with solution $x'_{i,j}$, we define: 
    \begin{equation*}
        x_{i,j} = \left\{ \begin{tabular}{cc}
             $\frac{v_j}{v_1+v_2} \cdot x'_{i,1}$ &  if $\, j \leq 2$ \\
             $x'_{i,j-1}$                         &  if $\, j \geq 3$ 
        \end{tabular}\right.
    \end{equation*}
    Then $x$ is a solution for $\pi$ satisfying \eqref{eq:frac_primal1}, \eqref{eq:frac_primal2} and \eqref{eq:frac_primal3} since:
    \begin{equation*}\renewcommand{\arraystretch}{1.5}\begin{tabular}{ll}
        $\sum_{j=1}^k x_{i,j} = 
            \left(\frac{v_1}{v_1+v_2} + \frac{v_2}{v_1+v_2}\right) x'_{i,1} + \sum_{j=2}^k x'_{i,j} = 1$   
                                                                        & for all $i \in [n]$,   \\
        $\sum_{i=1}^n u_i x_{i,j}  = 
            \frac{v_j}{v_1+v_2} \sum_{i=1}^n u_i x'_{i,1} = v_j$        & for $j =1,2$, \\
        $\sum_{i=1}^n u_i x_{i,j}  = 
            \sum_{i=1}^n u_i x'_{i,j-1} = v_j$                          & for $j = 3,\ldots,k$, \\
        $\sum_{i=1}^n a_i x_{i,j}  = 
            \frac{v_j}{v_1+v_2} \sum_{i=1}^n a_i x'_{i,j} = \frac{v_j}{v_1+v_2} \cdot (b_1+b_2) = b_j$
                                                                        & for $j =1,2$, \\
        $\sum_{i=1}^n a_i x_{i,j}  = 
            \sum_{i=1}^n a_i x'_{i,j-1} = b_j$                          & for $j = 3,\ldots,k$,
    \end{tabular}\end{equation*}
    where the equality $\frac{v_j}{v_1+v_2} \cdot (b_1+b_2) = b_j$ for $j =1,2$ follows from the assumption that $b_1/v_1 = b_2/v_2$.
    To prove the claim regarding the slack, observe that $V_l = V'_{l-1}$ for any $l \geq 2$. 
    Therefore, for all $i \in [n]$ and $l=2,\ldots k$ we have:
    \begin{equation*}
        \mu_{i,l} = \max(0,\min(u_i,V_l-U_{i-1})) = \max(0,\min(u_i,V'_{l-1}-U_{i-1})) = \mu'_{i,l-1}
    \end{equation*}
    It follows that for all $l=2,\ldots,k$:
    \begin{equation*}
        \slack_l(\pi) 
        = \sum_{i=1}^n \frac{a_i}{u_i} \cdot \mu_{i,l} - \sum_{j=1}^l b_j 
        = \sum_{i=1}^n \frac{a_i}{u_i} \cdot \mu'_{i,l-1} - \sum_{j=1}^{l-1} b'_j = \slack_{l-1}(\pi'),
    \end{equation*}
    as claimed.
\end{proof}

\begin{lem}\label{lem:pour_from_first}
    Consider an instance $\pi=(\overline{a}, \overline{u}, \overline{b}, \overline{v})$ of the fluid mixing problem satisfying the slack condition,
    with $n \ge 1$, $k \ge 2$ and with strictly positive vectors $\overline{u}$ and $\overline{v}$, such that $\frac{b_1}{v_1} > \frac{b_2}{v_2}$. 
    Let $\lambda = \min\!\left( 1,\, \frac{v_1}{u_1},\, \frac{b_1 v_2 \,-\, v_1 b_2}{a_1 v_2 \,-\, u_1 b_2} \right)$,  and let $\pi'=(\overline{a'}, \overline{u'}, \overline{b'}, \overline{v'})$ be the problem obtained by pouring a portion $\lambda$ of the liquid from the first source container into the first target container, where
    \begin{equation*}\begin{tabular}{ll}
        $\overline{a}' = ( a_1 - \lambda a_1, a_2, \ldots, a_n )$, & 
            $\quad \overline{u}' = ( u_1 - \lambda u_1, u_2, \ldots, u_n )$, \\
        $\overline{b}' = ( b_1 - \lambda a_1, b_2, \ldots, b_k )$, &
            $\quad \overline{v}' = ( v_1 - \lambda u_1, v_2, \ldots, v_k )$.
    \end{tabular}\end{equation*}
    
    Then $\lambda>0$, and $\pi'$ is a valid problem with $\slack_l(\pi')=\slack_l(\pi)$ for $l \in [k]$, and if $\pi'$ is solvable so is $\pi$.  
    If $X'$ is a solution of $\pi'$ then 
    \[
        X = \begin{bmatrix} \lambda(\delta_1)_{1 \times k} + (1-\lambda)X'_{1,1:k}\\[4pt] X'_{2:n,1:k} \end{bmatrix}
    \]
    is a solution for $\pi$.
    If $\slack_1(\pi) > 0$, then $\lambda < \frac{v_1}{u_1}$.
\end{lem}

\begin{proof}
    The proof is established as a sequence of claims.
    \begin{claim}\label{claim:pour1}
        $\frac{a_1}{u_1} \ge \frac{b_1}{v_1}$
    \end{claim}
    \begin{claimproof}
        \begin{equation}\label{eq:positive_slack1}
            0 \le \slack(\pi) \le \slack_1(\pi) = \sum_{i=1}^n \frac{a_i}{u_i} \mu_{i,1} - b_1 \le \frac{a_1}{u_1} \sum_{i=1}^n \mu_{i,1} - b_1 = \frac{a_1}{u_1} v_1 - b_1.            
        \end{equation}
    \end{claimproof}

    \noindent
    Denote $\lambda_0=\frac{b_1 v_2 \,-\, v_1 b_2}{a_1 v_2 \,-\, u_1 b_2}$.

    \begin{claim}\label{claim:pour2}
     $\lambda_0>0$, and thus, $\lambda>0$.
    \end{claim}
    \begin{claimproof}
        In the expression for $\lambda_0$, both the numerator and the denominator are strictly positive, as by Claim~\ref{claim:pour1} and the assumption of the lemma, we have $\frac{a_1}{u_1} \ge \frac{b_1}{v_1} > \frac{b_2}{v_2}$. Since $\lambda = \min\!\left( 1,\, \frac{v_1}{u_1},\, \lambda_0 \right)$, the result follows.
    \end{claimproof}

    \begin{claim}\label{claim:pour_valid1}
        $a'_1, u'_1, b'_1, v'_1 \ge 0$
    \end{claim}
    \begin{claimproof}
        The non-negativity of $a'_1$ and $u'_1$ follows from the fact that $\lambda \le 1$, and the non-negativity of $v'_1$ follows from $\lambda \le \frac{v_1}{u_1}$.
        In order to prove that $b'_1 \ge 0$, we use the fact that $\lambda \le \lambda_0$ to obtain that
        \[
            b'_1 = b_1 - \lambda a_1 \ge b_1 - \lambda_0 a_1 
                 = b_1 - \frac{b_1 v_2 \,-\, v_1 b_2}{a_1 v_2 \,-\, u_1 b_2} a_1
                 = \frac{a_1 v_1 - u_1 b_1}{a_1 v_2 \,-\, u_1 b_2} b_2 \ge 0,
        \]
    where the last inequality follows from Claim~\ref{claim:pour1} and the assumption of the lemma.
    \end{claimproof}
    \begin{claim}\label{claim:pour_valid2}
        If $u'_1=0$ then $a'_1=0$; if $v'_1=0$ then $b'_1=0$.
    \end{claim}
    \begin{claimproof}
        If $u'_1=0$ then $\lambda = 1$, and therefore $a'_1=0$ as well.
        If $v'_1=0$ then $\lambda = \frac{v_1}{u_1}$. 
        Since also we have, $\lambda \le\lambda_0=\frac{b_1 v_2 \,-\, v_1 b_2}{a_1 v_2 \,-\, u_1 b_2}$, it follows that
        \[
            v_1 (a_1 v_2 - u_1 b_2) \le u_1(b_1 v_2 - v_1 b_2).
        \]
        Simplifying, yields that this requirement is equivalent to $\frac{a_1}{u_1} \le \frac{b_1}{v_1}$, which by Claim~\ref{claim:pour1} implies that $\frac{a_1}{u_1} = \frac{b_1}{v_1}$.
        Therefore,
        \[
            b'_1 = b_1 - \lambda a_1 = b_1 - \frac{v_1}{u_1} a_1 = \frac{b_1 u_1 - a_1 v_1}{u_1} = 0.
        \]
    \end{claimproof}

    \begin{claim}\label{claim:pour_valid3}
        $\frac{b'_1}{v'_1} \ge \frac{b_2}{v_2}$
    \end{claim}
    \begin{claimproof}
        As by Claim~\ref{claim:pour1}, $\frac{a_1}{u_1} \ge \frac{b_1}{v_1}$ we conclude that $\frac{b'_1}{v'_1} = \frac{b_1 - \lambda a_1}{v_1 - \lambda u_1}$ is monotone non-increasing with $\lambda$. In conjunction with, $\lambda \le \lambda_0 = \frac{b_1 v_2 \,-\, v_1 b_2}{a_1 v_2 \,-\, u_1 b_2}$, it follows that:
        \[
            \frac{b'_1}{v'_1} 
            = \frac{b_1 - \lambda a_1}{v_1 - \lambda u_1} 
            \ge \frac{b_1 - \lambda_0 a_1}{v_1 - \lambda_0 u_1}
            = \frac{b_1 (a_1 v_2 - u_1 b_2) - a_1 (b_1 v_2 - v_1 b_2)}{v_1 (a_1 v_2 - u_1 b_2) - u_1 (b_1 v_2 - v_1 b_2)}
            = \frac{(a_1 v_1 - b_1 u_1) b_2}{(v_1 a_1 - u_1 b_1) v_2}
            = \frac{b_2}{v_2}.
        \]
    \end{claimproof}
    
    \begin{claim}\label{claim:pour_valid}
        The problem $\pi'$ is valid.
    \end{claim}
    \begin{claimproof}
        By definition, it satisfies the conservation equality \eqref{eq:conservation_Of_volume_and_mass}. 
        The vectors $\overline{a'}, \overline{u'}, \overline{b'}, \overline{v'}$ are non-negative, by the non-negativity of $\overline{a}, \overline{u}, \overline{b}, \overline{v}$ and Claim~\ref{claim:pour_valid1}.
        Also, a zero volume source or target container implies zero mass by Claim~\ref{claim:pour_valid2}.
        Finally, the source and target container ordering of $\pi'$ is valid. This follows from the order validity in $\pi$ and the fact that unless the first source container becomes empty, its density does not change.
        As for the target containers, if the first target container does not become full, the claim follows by Claim~\ref{claim:pour_valid3} and by the ordering in $\pi$.
    \end{claimproof}
        
    \begin{claim}\label{claim:pour6}
        If $\slack_1(\pi) > 0$, then $\lambda < \frac{v_1}{u_1}$.
    \end{claim}    
    \begin{claimproof}
        By \eqref{eq:positive_slack1}, $\slack_1(\pi) > 0$ implies $\frac{a_1}{u_1} > \frac{b_1}{v_1}$.
        On the other hand, if $\lambda = \frac{v_1}{u_1}$, by the argument of Claim~\ref{claim:pour_valid2}, implies that $\frac{a_1}{u_1} \le \frac{b_1}{v_1}$, which is a contradiction.
    \end{claimproof}
        
    Next, we prove that given the solution $X'$ for $\pi'$ implies that $X = [x_{i,j}]_{i \in [n], j ]in [k]}$ is a solution for $\pi$:
    \begin{equation*}
        x_{i,j} = \begin{cases}
            x'_{i,j}                       & \textrm{if }i > 1 \\
            (1-\lambda)x'_{1,j}            & \textrm{if }i = 1 \textrm{ and }j > 1 \\
            (1-\lambda)x'_{1,j} + \lambda  & \textrm{if }i = j = 1
        \end{cases}
    \end{equation*} 
    This we do by checking that the conditions of Definition~\ref{def:fluid_mixing} are met by $X$, where the non-negativity of $X$ immediately follows from the non-negativity of $X'$ and $\lambda \le 1$.
    We prove that constraints \eqref{eq:frac_primal1}, \eqref{eq:frac_primal2} and \eqref{eq:frac_primal3} hold in the next three claims.

    \begin{claim}\label{claim:pourX1}
        condition \eqref{eq:frac_primal1} holds for $X$.
    \end{claim}
    \begin{claimproof}
        The constraint holds for $i>1$ as $x_{i,j}=x'_{i,j}$ for all $j$.
        Otherwise, $i=1$ and then
        \begin{equation*}
            \sum_{j=1}^k x_{1,j} = \lambda + (1-\lambda)\sum_{j=1}^k x'_{1,j} = \lambda + (1-\lambda) = 1.
        \end{equation*}
    \end{claimproof}
    
    \begin{claim}\label{claim:pourX2}
        condition \eqref{eq:frac_primal2} holds for $X$.
    \end{claim}
    \begin{claimproof}
        The proof proceeds by cases.
        If $\lambda<1$ and $j>1$, then
        \begin{align*}
            \sum_{i=1}^n u_i x_{i,j} 
               = \frac{1}{1-\lambda} u'_1 \cdot (1-\lambda)x'_{1,j}\ + \sum_{i=2}^n u'_i x'_{i,j} =\sum_{i=1}^n u'_i x'_{i,j} = v'_j = v_j.
        \end{align*}
        Else, if $\lambda<1$ and $j=1$, then
        \begin{align*}
            \sum_{i=1}^n u_i x_{i,1} 
               = u_1 \cdot \left((1-\lambda)x'_{1,1} + \lambda\,\right) + \sum_{i=2}^n u'_i x'_{i,1} 
                = \lambda u_1 + \sum_{i=1}^n u'_i x'_{i,1} = \lambda u_1 + v'_1 = v_1.
        \end{align*}
        Else, if $\lambda=1$ and $j>1$, then $u'_1=0$ and we have
        \begin{align*}
            \sum_{i=1}^n u_i x_{i,j}= \sum_{i=2}^n u'_i x'_{i,j} = \sum_{i=1}^n u'_i x'_{i,j} = v'_j = v_j.
       \end{align*}
       Else, we have $\lambda=1$ and $j=1$, and then
        \begin{align*}
            \sum_{i=1}^n u_i x_{i,1}
               = u_1 \cdot 1 + \sum_{i=2}^n u'_i x'_{i,1} = \lambda u_1 + v'_1 = v_1.
       \end{align*}
    \end{claimproof}
    
    \begin{claim}\label{claim:pourX3}
        condition \eqref{eq:frac_primal3} holds for $X$.
    \end{claim}
    \begin{claimproof}
        The proof proceeds by cases.
        If $\lambda<1$ and $j>1$, then
        \begin{align*}
            \sum_{i=1}^n a_i x_{i,j} = \frac{1}{1-\lambda} a'_1 \cdot (1-\lambda)x'_{1,j} + \sum_{i=2}^n a'_i x'_{i,j} = \sum_{i=1}^n a'_i x'_{i,j} = b'_j = b_j.
        \end{align*}
        Else, if $\lambda<1$ and $j=1$, then
        \begin{align*}
            \sum_{i=1}^n a_i x_{i,1} 
               = a_1 \cdot \left((1-\lambda)x'_{1,1} + \lambda \right) + \sum_{i=2}^n a'_i x'_{i,1} 
               = \lambda a_1 + \sum_{i=1}^n a'_i x'_{i,1} = \lambda a_1 + b'_1 = b_1.
        \end{align*}
        Else, if $\lambda=1$ and $j>1$, then $a'_1=0$ and we have
        \begin{align*}
            \sum_{i=1}^n a_i x_{i,j} = \sum_{i=2}^n a'_i x'_{i,j} = \sum_{i=1}^n a'_i x'_{i,j} = b'_j = b_j.
       \end{align*}
        Else, we have $\lambda=1$ and $j=1$, and then
        \begin{align*}
            \sum_{i=1}^n a_i x_{i,1} = a_1 \cdot 1 + \sum_{i=2}^n a'_i x'_{i,1} = \lambda a_1 + b'_1 = b_1.
       \end{align*}
    \end{claimproof}

    \begin{claim}
            With $\mu_{i,l}$,  $\mu'_{i,l}$ given by Definition~\ref{def:frac_slack} for $i \in [n]$ and $l \in [k]$, we have
           \[
                \mu'_{i,l} = \begin{cases}
                    \mu_{i,l} - \lambda u_1 & \text{ if } i=1 \\
                    \mu_{i,l}               & \text{ if } i>1
                \end{cases}
            \]
    \end{claim}
    \begin{claimproof}
       If $i=1$:
       \begin{align*}
            \mu'_{1,l} 
                &= \max(0,\min(u'_1,V'_l-U'_0))
                = \max(0,\min(u_1-\lambda u_1,V_l-\lambda u_1))\\
                &= \max\left(0, \min(u_1,V_l) - \lambda u_1\right)
                 = \min(u_1,V_l) - \lambda u_1 \\
                 &= \max( 0, \min(u_1,V_l-U_0) ) - \lambda u_1 
                = \mu_{1,l} - \lambda u_1,
       \end{align*}
       where the inequality $\min(u_1,V_l-U_0) \geq \lambda u_1$ holds as $\lambda \leq \frac{v_1}{u_1} \leq \frac{V_l}{u_1}$.
       Otherwise, $i>1$, and then
       \begin{align*}
            \mu'_{i,l} 
                &= \max(0,\min(u'_i,V'_l-U'_{i-1}))
                = \max\!\big(0,\min(u_i,(V_l-\lambda u_1)-(U_{i-1}-\lambda u_1))\big)\\
                &= \max(0,\min(u_i,V_l-U_{i-1}))
                = \mu_{i,l}.
       \end{align*}
    \end{claimproof}
    
    \begin{claim}
        $\slack(\pi') = \slack(\pi)$
    \end{claim}
    \begin{claimproof}
       If $\lambda<1$, then
       \begin{align*}
            \slack_l(\pi')
                &= \sum_{i=1}^n \frac{a'_i}{u'_i} \cdot \mu'_{i,l} - \sum_{j=1}^l b'_j 
                 = \frac{a_1(1-\lambda)}{u_1(1-\lambda)} \cdot (\mu_{1,l}-\lambda u_1) + \sum_{i=2}^n \frac{a_i}{u_i} \cdot \mu_{i,l} - (b_1 - \lambda a_1) - \sum_{j=2}^l b_j \\
                &= -\frac{a_1}{u_1} \cdot \lambda u_1 + \sum_{i=1}^n \frac{a_i}{u_i} \cdot \mu_{i,l} - \sum_{j=1}^l b_j + \lambda a_1
                = \slack_l(\pi).
       \end{align*}
       Otherwise, $\lambda=1$. We note that $u'_1=\mu'_{1,l}=0$ and that the term $\frac{a'_1}{u'_1} \mu'_{1,l}$ in $\slack_l(\pi')$ can be ignored. 
       Therefore:
       \begin{align*}
            \slack_l(\pi')
                = \sum_{i=2}^n \frac{a'_i}{u'_i} \cdot \mu'_{i,l} - \sum_{j=1}^l b'_j
                = \sum_{i=2}^n \frac{a_i}{u_i} \cdot \mu_{i,l} - \sum_{j=1}^l b_j + a_1.
       \end{align*}
       On the other hand, $1 = \lambda \leq \frac{v_1}{u_1}$, implying that $V_l \geq v_1 \geq u_1$.
       It follows that $\mu_{1,l} = \max(0,\min(u_1,V_l-U_0)) = u_1$ and therefore
       \begin{align*}
            \slack_l(\pi)
                = \sum_{i=1}^n \frac{a_i}{u_i} \cdot \mu_{i,l} - \sum_{j=1}^l b_j
                = a_1 + \sum_{i=2}^n \frac{a_i}{u_i} \cdot \mu_{i,l} - \sum_{j=1}^l b_j = \slack_l(\pi').
       \end{align*}        
    \end{claimproof}
    This concludes the proof of of Lemma~\ref{lem:pour_from_first}.
\end{proof}

The following lemma highlight details of the solution constructed in the proof of \Cref{thm:slack_implies_fluid_mixing}, which will be useful in the next section. 

\begin{lem}\label{lem:fluid_mixing_solution_structure}
    Let $\pi$ be an instance of the fluid mixing problem with $\slack(\pi) > 0$.
    Let $X = [x_{i,j}]_{i \in [n],j \in [k]}$ be the solution constructed by Algorithm~\ref{alg:fluid_mixing_solution_rec} in the proof of Theorem~\ref{thm:slack_implies_fluid_mixing}. 
    We denote by $\gamma_0,\ldots,\gamma_{k-1}$ the sequence of values assigned to the variable $\gamma$ in line~\ref{algline:fluid_mixing_solution_rec:merge_gamma} of \Cref{alg:fluid_mixing_solution_rec},
    indexed by their order of execution, except for the dummy variable $\gamma_0=0$. 
    Specifically, $\gamma_j$ for $j>0$ corresponds to the iteration where the subproblem $\pi'$ contains exactly $k-j$ target containers.
    Also, we define $I_j = \min \{i\in [n] : x_{i,j}>0\}$ for all $j \in [k]$, and $I_{k+1} = n+1$. 
    Then
    \begin{enumerate}
        \item $I_1=1$ and $I_j$ is monotone non-decreasing in $j$,
        \item $x_{i,j}>0$ for all $j \in [k]$ and $i \geq I_j$,
        \item For $l \in [k]$ and $i$ such that $I_l < i < I_{l+1}$, we have $x_{i,j} = \begin{cases}
            (1-\gamma_{j-1}) \prod_{j'=j}^{l-1} \gamma_{j'}& 1 \le j \le l \\
            0                                              & l < j \le k
        \end{cases}$.
    \end{enumerate}
\end{lem}

Note that statement (3) gives us information only on rows $i$ for which $i \neq I_j$ for all $j$, so the number of excluded rows is at most $k$.
Therefore, if $k \ll n$, as will be the case in the next section, this lemma gives information on almost all rows of the matrix $X$.
Also note that (3) implies that $x_{\cdot,j}$ is constant in blocks. 
That is, $x_{i+1,j} = x_{i,j}$ for all $j \in [k]$ and $i \in [n-1]$ such that $i, i+1 \not \in \{I_2,I_3,\ldots,I_k\}$.

\begin{proof}
    We prove the claim by induction on $n+k$, where the base case of $k=1$ trivially holds.
    Given a problem $\pi$ with $k>1$ and $\slack(\pi) > 0$, 
    Theorem~\ref{thm:slack_implies_fluid_mixing} proceeds by applying one of the four lemmas \ref{lem:null_u}, \ref{lem:null_v}, \ref{lem:merge_v} and \ref{lem:pour_from_first} to obtain a smaller problem $\pi'$, 
    where each of the four lemmas, guaranties that a positive $\slack(\pi)$ implies a positive $\slack(\pi')$.
    The positive slack implies, by Lemma~\ref{lem:pour_from_first}, that no empty target container is ever formed, so Lemma~\ref{lem:null_v} is never used.

    We show that conditions (1) through (3) of the proposition hold for each of the three transformations from a solution $X'$ for $\pi'$ into a solution $X$ for $\pi$, where we let $I'$ and $\gamma'$ denote the $I$ and $\gamma$ vectors of  $X'$. Note that condition (3) includes the case $l=i=1$ if $I_2>1$.
    \begin{itemize}
        \item
            The transformation of Lemma~\ref{lem:null_u} for handling a null source container is to prepend the row $[1, 0, \cdots, 0]$ to the matrix $X'$. Then $\gamma=\gamma'$, $I_1=1$ and $I_j=I'_j+1$ for $j > 1$, so all three conditions hold. 
        \item
            The transformation of Lemma~\ref{lem:merge_v} for merging two target containers of the same density, for some value $0 < \gamma < 1$,
            is $X = \begin{bmatrix}\gamma X'_{1:n,1} & (1-\gamma)X'_{1:n,1} & X'_{1:n,2:k}\end{bmatrix}$.
            Then, $I_1=I_2=I'_1$ and $I_j=I'_{j-1}$ for $j>2$; and $\gamma_0=0$, $\gamma_1=\gamma$, and $\gamma_j = \gamma'_{j-1}$ for $j \ge 2$, 
            so all three conditions hold.
        \item 
            The transformation of Lemma~\ref{lem:pour_from_first} for pouring portion $0 < \lambda \le 1$ from the first source container to the first target container is to replace the first row by
            $(1-\lambda) (x'_{1, \cdot}) + \lambda[1, 0, \cdots, 0]$.
            Then both $I$ and $\gamma$ vectors stay the same, and the claim follows.
    \end{itemize}
\end{proof}

The next definition and lemma show that for a certain class of problems, that we call {\em $\epsilon$-robust}, all entries of the solution $X$ constructed by \Cref{alg:fluid_mixing_solution_rec}, from row $I_k+1$, are uniformly bounded away from zero by some constant $\delta$ that is only a function of $\epsilon$ and $k$, but is otherwise independent of the problem $\pi$. This property will enable us in the next section to derive an algorithm for solving a class of robust partition problems, which we call {\em linear}.

\begin{defn}\label{def:robust_fluid_mixing}
    Given a fluid mixing problem $\pi=(\overline{a}, \overline{u}, \overline{b}, \overline{v})$, we say that it is $\epsilon$-robust if the following four conditions hold, where the total volume is $V=\sum_{j=1}^k v_j$ and the total mass is $M=\sum_{j=1}^k b_j$:
    \begin{enumerate}
        \item $\slack(\pi) > \epsilon M$,
        \item $b_j > \epsilon M$ for $j \in [k]$,
        \item $v_j > \epsilon V$ for $j \in [k]$,
        \item $\frac{M}{V} > \epsilon \frac{a_1}{u_1}$.
    \end{enumerate}
\end{defn}

\begin{lem}\label{lem:non_vanishing_last_block}
    Let $\pi=(\overline{a}, \overline{u}, \overline{b}, \overline{v})$ be an $\epsilon$-robust fluid mixing problem, for some $\epsilon>0$.
    Then there is a constant $\delta>0$, depending only on $\epsilon$ and $k$, so that the solution $X$ constructed by Algorithm~\ref{alg:fluid_mixing_solution_rec} satisfies
    \[
        x_{ij} > \delta \text{ for all } j \in [k] \text{ and } i > I_k. 
    \]
\end{lem}

\begin{proof}
    Given $\pi$ satisfying the requirements of the lemma, we prove that the values $\gamma_1,\ldots,\gamma_{k-1}$, as defined in Lemma~\ref{lem:fluid_mixing_solution_structure},
    satisfy $\delta_0 < \gamma_j < 1 - \delta_0$ for all $j \in [k-1]$, where $\delta_0 > 0$ is a constant depending only on $\epsilon$ and $k$. 
    Proving this would imply by Lemma~\ref{lem:fluid_mixing_solution_structure} (3) that the required result follows with $\delta = (\delta_0)^k$.

    The proof proceeds by induction on $k$, where the induction hypothesis is that if $\pi$ is an $\epsilon$-robust fluid mixing problem with $\epsilon>0$, then 
    in its solution $\delta_0 < \gamma_j < 1 - \delta_0$ for for all $j \in [k-1]$, where $\delta_0 > 0$ is a constant depending only on $\epsilon$ and $k$.
    As the proof trivially holds for $k=1$, we assume that $\pi$ is an $\epsilon$-robust fluid mixing problem with $k>1$ and that the claim holds for all smaller values of $k$.

    The proof idea is to apply the induction hypothesis to the problem $\pi'$ obtained by the first target containers merge. 
    As for the problem $\pi$, all $\gamma_j$ for $j>1$ are inherited from $\pi'$, and it remains to show that:
    \begin{enumerate}
        \item $\gamma_1$ is bounded away from both zero and one, 
        \item $\pi'=(\overline{a'}, \overline{u'}, \overline{b'}, \overline{v'})$ is $\epsilon'$-robust for some $\epsilon'>0$ that depends only on $\epsilon$ and $k$.
    \end{enumerate}
    Let $x_1$ denote the amount of source liquid poured into the first target container (possibly from multiple source containers) until the first two target containers merge.
    The upper bound on $\gamma_1$ is
    \[
        \gamma_1 = \frac{v_1-x_1}{(v_1-x_1) + v_2} \le \frac{V}{V+v_2} \le \frac{1}{1+\epsilon}.
    \]
    In order to obtain the lower bound on $\gamma_1$, we define the function $f(x)$ as the maximum possible source liquid mass obtainable for volume $x$.
    Formally, $f(x) = \sum_{i=1}^n \nu_i(x) \frac{a_i}{u_i}$ with $\nu_i(x) = \max(0, \min(u_i, x - U_{i-1}))$, where as before $U_i$ is the prefix sum of the vector $u_i$. We also define the function $g(x)$ as the required density of the first target container after pouring volume $x$ of the densest source liquids,
    \[
        g(x) = \frac{b_1 - f(x)}{v_1 - x}.
    \]
    It should be noted that $\slack_1(\pi) = f(v_1) - b_1 > 0$, implying that the range of $x$ for which the numerator of $g(x)$ is non-negative is $[0,x_0]$ for some value $0 < x_0 < v_1$. Also, we note that the function $f$ is piecewise linear, as its slope for $x \in [U_{i-1},U_i]$ is $a_i/u_i$. Moreover, it is concave as $a_i/u_i$ is a decreasing sequence.

    \begin{claim}
        The function $g(x)$ is monotone decreasing for $x \in [0,x_0]$
    \end{claim}
    \begin{claimproof}
        We prove piecewise monotonicity for $x$ in the interval $[0,x_0]$.
        For numbers $x,\Delta$ and $l$ such that $U_{l-1} \le x < x+\Delta \le \min(U_l,x_0)$, we have
        \[
            g(x + \Delta) = \frac{b_1 - f(x + \Delta)}{v_1 - x - \Delta} = \frac{b_1 - f(x) - \Delta \frac{a_l}{u_l}}{v_1 - x - \Delta}.
        \]
        Then, proving that $g(x + \Delta) \le g(x)$ amounts to proving that 
        \[
            \frac{a_l}{u_l} \ge \frac{b_1 - f(x)}{v_1 - x}.
        \]
        Assume to the contrary that for some $x,\Delta$ and $l$, the required inequality does not hold, $b_1 - f(x) > \frac{a_l}{u_l} (v_1 - x)$. 
        Then
        \begin{align*}
            b_1 
              &< f(v_1) 
              =  \sum_{i=1}^n \nu_i(v_1) \frac{a_i}{u_i}
              =  \sum_{i=1}^n \Big(\nu_i(v_1) - \nu_i(x)\Big) \frac{a_i}{u_i} + f(x)
              =   \sum_{i=l}^n \Big(\nu_i(v_1) - \nu_i(x)\Big) \frac{a_i}{u_i} + f(x) \\
              &\le \frac{a_l}{u_l} \sum_{i=l}^n \Big(\nu_i(v_1) - \nu_i(x)\Big) + f(x) 
              = \frac{a_l}{u_l} \Big( \sum_{i=1}^n \nu_i(v_1) - \sum_{i=1}^n \nu_i(x) \Big) + f(x)
              =   \frac{a_l}{u_l} (v_1 - x)  + f(x)
              < b_1,
        \end{align*}
        which is a contradiction.
    \end{claimproof}    

    \begin{claim}
        The equation $g(x) = \frac{b_2}{v_2}$ has a solution $x_1 \in [0,x_0]$, and the solution satisfies $v_1 - x_1 \ge \epsilon^2 V$.
    \end{claim}
    \begin{claimproof}
        The equation has a solution $x_1 \in [0,x_0]$, since $g$ is continuous and $g(0)=\frac{b_1}{v_1} \ge \frac{b_2}{v_2}$ and $g(x_0) = 0 < \frac{b_2}{v_2}$. 

        The $\epsilon$-robustness of $\pi$ implies that $f(v_1) - b_1 = \slack_1(\pi) \ge \epsilon M$ and that $\frac{a_1}{u_1} < \frac{M}{\epsilon V}$. Therefore,
        \begin{align*}
            \epsilon M 
              \le f(v_1) - b_1
              =   f(v_1) - f(x_0)
              \le f(v_1) - f(x_1)
              = \sum_{i=1}^n \Big(\nu_i(v_1) - \nu_i(x_1)\Big) \frac{a_i}{u_i}
              \le \frac{a_1}{u_1} (v_1 - x_1)
              < \frac{M}{\epsilon V} (v_1 - x_1),
        \end{align*}
        implying that $v_1 - x_1 \ge \epsilon^2 V$, as claimed.
    \end{claimproof}
    This implies the lower bound on $\gamma_1$ as
    \[
        \gamma_1 = \frac{v_1-x_1}{(v_1-x_1) + v_2} \ge \frac{\epsilon^2 V}{\epsilon^2 V + v_2} \ge \frac{\epsilon^2 V}{\epsilon^2 V + V} = \frac{\epsilon^2}{\epsilon^2+1}.
    \]
    It remains to show that the problem $\pi'$ is $\epsilon'$-robust for some $\epsilon'>0$ that is only a function of $\epsilon$ and $k$.
    We check the four robustness conditions:
    \begin{enumerate}
        \item $\slack_j(\pi')=\slack_{j+1}(\pi) > \epsilon M \ge \epsilon M'$ for $j \in [k-1]$.  
        \item $b'_j = b_{j+1} > \epsilon M \ge \epsilon M'$ for $j=2,\ldots,k-1$ and $b'_1 \ge b_2 > \epsilon M \ge \epsilon M'$.
        \item $u'_j = u_{j+1} > \epsilon V \ge \epsilon V'$ for $j=2,\ldots,k-1$ and $u'_1 \ge u_2 > \epsilon V \ge \epsilon V'$,
        \item 
            $\epsilon \frac{a'_1}{u'_1} \le \epsilon \frac{a_1}{u_1} \le \frac{M}{V} \le \frac{M}{V'} \le \frac{1}{\epsilon}\frac{M'}{V'}$,
            where in the last inequality, we use the inequality $\epsilon M \le b_k \le M'$.
    \end{enumerate}
\end{proof}

%%%%%%%%%%%%%%%%%%%%%%%%%%%%%%%%%%%%%%%%%%%%%%%%%%%%%%%%%%%%%%%%%%%%%%
\section{The Linear Partition Problem}\label{sec:linear}

In this section, we return to the original partition problem and consider instances where the part sizes $p_1,\ldots,p_k$ are linear functions of $n$, namely $p_i = \alpha_i n$. We show that if these sizes satisfy a stronger version of the slack condition, the problem is solvable for a sufficiently large $n$ using a randomized algorithm whose failure probability decays exponentially with $n$.
The solution is based on applying randomized rounding to the fractional solution of the linear problem presented in the previous section and showing that the imperfection introduced by the rounding can be mended with high probability by reassigning a small number of elements.
The uniform lower bound $\delta$ on the entries of the fractional solution $X$, supplied by Lemma~\ref{lem:non_vanishing_last_block}, plays a crucial role in performing the corrections, as it guarantees, with high probability, the existence of an abundant pool of swap candidates for equating the set sums.

\begin{defn}
   Let $A=[\alpha_1, \ldots ,\alpha_k]$ be a non-descending sequence of positive rational numbers satisfying $\sum_{i=1}^k \alpha_i = 1$.
   The {\em linear partition problem family $\PP_{\alpha_1,\ldots,\alpha_k}$} is defined as the set of equal sum partition problems $(n,k,P_n=[\alpha_1 n,\ldots,\alpha_k n])$, where $n$ is a positive integer such that $\alpha_j n$ are integral for all $j=1,\ldots,k$, and 
    \begin{equation}\label{eq:asymp_slack}
         \lim_{n \rightarrow \infty}\frac{1}{n^2}\slack_j(P_n) > 0 \text{ for all }j=1,\ldots,k-1.
    \end{equation}
\end{defn}

\begin{rem}
    Using the equality $s^{n,k} = \frac{n(n+1)}{2k}$ and defining $A_j = \sum_{j'=1}^j \alpha_{j'}$ for $j \in [k]$, the asymptotic slack condition \eqref{eq:asymp_slack} translates into:
    \begin{equation*}
         \lim_{n \rightarrow \infty}\frac{1}{n^2}\slack_j(P_n) 
            = \lim_{n \rightarrow \infty}\frac{1}{n^2} \left[\sum_{i=1}^{A_j n} (n-i+1)- j s^{n,k}\right] 
            = A_j\left(1-\frac{A_j}{2}\right) - \frac{j}{2k} > 0.
    \end{equation*}
\end{rem}

\begin{defn}\label{def:slack_alphas}
    Let $A=[\alpha_1, \ldots ,\alpha_k]$ be a non-descending sequence of positive rational numbers such that $\sum_{i=1}^k \alpha_i = 1$.
    Denote $A_j = \sum_{j'=1}^j \alpha_{j'}$ for $j \in [k]$ and define
    \begin{equation*}
        \slack_j(A) := A_j\left(1-\frac{A_j}{2}\right) - \frac{j}{2k}.
    \end{equation*}
    and
    \begin{equation*}
        \slack(A) := \min_{1\le j<k} \slack_j(A).
    \end{equation*}
    
\end{defn}

\noindent Thus, the asymptotic slack condition \eqref{eq:asymp_slack} is equivalent to $\slack(A) > 0$.

\medskip

\noindent Throughout this section, we let $S(Q)$ denote the sum of elements in the set $Q$: $S(Q) = \sum_{i \in Q} i$.

\medskip

\noindent We restate Theorem~\ref{thm:alg_solves_linear_partition_whp}:

\medskip

\noindent\textbf{Theorem \ref{thm:alg_solves_linear_partition_whp}.}
\emph{Let $\PP_{\alpha_1,\ldots,\alpha_k}$ be a linear partition problem family. 
Then, there exist $\delta>0$ and $N \in \N$ so that for every $n \ge N$, the probability that Algorithm~\ref{alg:linear_partition} fails to solve $(n,k,P_n) \in \PP_{\alpha_1,\ldots,\alpha_k}$ is at most $e^{-\delta n}$.}

\medskip

\begin{cor}\label{cor:linear_partition_is_solvable}
    A linear partition problem family $\PP_{\alpha_1,\ldots,\alpha_k}$ is solvable for a sufficiently large value of $n$.
\end{cor}

\begin{algorithm}[H]
\caption{Randomized Rounding For Linear Partitions}\label{alg:linear_partition}
\begin{algorithmic}[1]
\Require a sequence of positive rational numbers $\alpha_1 \leq \ldots \leq \alpha_k$ satisfying $\sum_{i=1}^k \alpha_i = 1$
\Require a natural number $n$ such that $\snk$ as well as $\alpha_j n$ for $j\in[k]$ are integers
\Ensure a partition $Q_1 \sqcup \cdots \sqcup Q_k = [n]$ where $|Q_j|=\alpha_j n$ and $S(Q_j) = \snk$ for all $j$ or Failure
\State define a fluid mixing problem $\pi=(\overline{a}, \overline{u}, \overline{b}, \overline{v})$, where $a_i=n-i+1$, $u_i=1$ for  $i\in[n]$, and $b_j=\snk$, $v_j=\alpha_j n$ for $j\in[k]$ \label{alg:define_problem}
\State let $x_{i,j}$ for $i\in[n]$, $j\in[k]$ be the solution of $\pi$ provided by Algorithm~\ref{alg:fluid_mixing_solution_rec}                                  \label{alg:problem_solution}
\State choose values $X_{i} \in [k]$ independently at random for $i \in [n]$, where $\Pr[X_{i}=j]=x_{i,j}$
\State let $Q_j \gets \{n-i+1: X_i=j\}$ for $j \in [k]$\label{alg:before_while_1}
\While{there exist $j',j'' \in [k]$ with $|Q_{j'}|>\alpha_{j'} n$ and $|Q_{j''}|<\alpha_{j''} n$}
    \State choose such $j',j''$ and choose $i \in Q_{j'}$
    \State let $Q_{j'} \gets Q_{j'} \setminus \{i\}$ and $Q_{j''} \gets Q_{j''} \cup \{i\}$
\EndWhile\label{alg:after_while_1}
\While{there exists $j$ where $S(Q_j) \neq \snk$}
    \State let $j' \gets \argmax_j S(Q_j)$ and $j'' \gets \argmin_j S(Q_j)$   \label{alg:2nd_while_j}
    \State let $\Delta \gets \min(S(Q_{j'}) - \snk, \snk - S(Q_{j''}))$
    \If {there exist $z' \in Q_{j'}$ and $z'' \in Q_{j''}$ such that $0 < z' - z'' \leq \Delta$}
        \State choose such $z',z''$ so that $z' - z''$ is maximal              \label{alg:2nd_while_z}
        \State let $Q_{j'} \gets (Q_{j'} \setminus \{z'\}) \cup \{z''\}$ and $Q_{j''} \gets (Q_{j''} \setminus \{z''\}) \cup \{z'\}$
    \Else
        \State \return Failure
    \EndIf
\EndWhile
\State \return $Q_1, \cdots, Q_k$
\end{algorithmic}
\end{algorithm}

\begin{rem}
    The use of Algorithm~\ref{alg:fluid_mixing_solution_rec} in Algorithm~\ref{alg:linear_partition} relies on the fact that its solution satisfies Lemma~\ref{lem:non_vanishing_last_block}.
    Standard, off-the-shelf LP solvers are generally unsuitable here, as they often yield a sparse solution with excessive zeros, which prevents randomized rounding from reaching its full potential. To circumvent this, the linear program can be modified to optimize directly for the required property, ensuring a solution that performs at least as well as that of Algorithm~\ref{alg:fluid_mixing_solution_rec}.
\end{rem}

\begin{proof}[Proof of Theorem~\ref{thm:alg_solves_linear_partition_whp}]
    Let $\PP_{\alpha_1,\ldots,\alpha_k}$ and $P_n$ be as in the statement. 
    We use the following version of Hoeffding's inequality (for example, \cite{alon2008probabilistic}, A.1.18).
    Given independent random variables $X_1, \ldots, X_n$ taking values in the segment $[a,b]$, let $X=\sum_{i=1}^n X_i$ and $\mu=E[X]$. Then:
    \begin{equation}\label{eq:Hoeffding}
        \Pr[|X-\mu| \ge t] < 2 \cdot \exp\!\left[-\frac{2t^2}{n(b-a)^2}\right]
    \end{equation}

    For variables that change value in the algorithm, such as $Q_j$, if it is not clear from the context, we specify the value at a give time as $Q_j^{(L\ref{alg:before_while_1})}$, which is the value just after line~\ref{alg:before_while_1} was executed, or $Q_j^{(l)}$. 
    
    \noindent
    The proof is presented as a sequence of claims. 
    
    \begin{claim}\label{claim:linear_v}
        For any $\epsilon > 0$, there exist $\delta>0$ and a natural number $N$, so that for every $n \ge N$, after line~\ref{alg:before_while_1} of Algorithm~\ref{alg:linear_partition},
        \[
            \Pr\!\left[ \max_{j \in [k]} \Big|\, |Q_j^{(L\ref{alg:before_while_1})}| - \alpha_j n\Big| > \epsilon n \right] < e^{-\delta n}.
        \]
    \end{claim}
    \begin{claimproof}
        For any $j \in [k]$, $|Q_j|$ is the sum of $n$ independent Bernoulli random variables, where the $i$-th variable equals $1$ with probability $x_{ij}$ and $E[\,|Q_j|\,] = \sum_{i=1}^n x_{ij} = \alpha_j n$. Therefore, by \eqref{eq:Hoeffding} with $t=\epsilon n$ and $b-a=1$,
        \[
            \Pr\!\left[ \big| |Q_j^{(L\ref{alg:before_while_1})}| - \alpha_j n\big| > \epsilon n \right] < 2 e^{-2 \epsilon^2 n}.
        \]
        Performing a union bound on $j \in [k]$ yields the required result for any $0 < \delta < 2\epsilon^2$ and sufficiently large $N$.
    \end{claimproof}

    \begin{claim}\label{claim:linear_m}
        For any $\epsilon > 0$, there exist $\delta>0$ and a natural number $N$, so that for every $n \ge N$, after line~\ref{alg:before_while_1} of Algorithm~\ref{alg:linear_partition}, 
        \[
            \Pr\!\left[ \max_{j \in [k]} \Big| S(Q_j^{(L\ref{alg:before_while_1})}) - \snk \Big| > \epsilon n^2 \right] < e^{-\delta n}.
        \]
    \end{claim}
    \begin{claimproof}
        For any $j \in [k]$, $S(Q_j^{(L\ref{alg:before_while_1})})$ is the sum of $n$ independent random variables where the $i$-th variable is $a_i$ with probability $x_{ij}$ and $0$ otherwise, where $E[S(Q_j^{(L\ref{alg:before_while_1})})] = \sum_{i=1}^n a_i x_{ij} = \snk$. By \eqref{eq:Hoeffding} with $t=\epsilon n^2$ and $b-a=n$, 
        \[
            \Pr\!\left[ \left| S(Q_j^{(L\ref{alg:before_while_1})}) - \snk\right| > \epsilon n^2 \right] \le 2 e^{-2\epsilon^2 n}.
        \]
        Performing a union bound on $j \in [k]$ yields the required result for any $0 < \delta < 2\epsilon^2$ and sufficiently large $N$.
    \end{claimproof}
    
    \begin{claim}\label{claim:linear_iter}
        Let $\iterI$ denote the number of iterations performed by the first while loop. Then, for any $\epsilon > 0$, there exist $\delta>0$ and a natural number $N$, so that for every $n \ge N$, 
        \[
            \Pr[\iterI > \epsilon n] < e^{-\delta n}.
        \]
    \end{claim}
    \begin{claimproof}
        We first observe that 
        \begin{equation}\label{eq:iter1}
            \iterI = \frac{1}{2}\sum_{j=1}^k \big| |Q_j^{(L\ref{alg:before_while_1})}| - \alpha_j n\big|
                   \le \frac{k}{2} \max_{j \in [k]} \Big|\, |Q_j^{(L\ref{alg:before_while_1})}| - \alpha_j n\Big|.
        \end{equation}
        Therefore, given $\epsilon>0$ we have 
        \[
            \Pr[\iterI > \epsilon n] \le \Pr\!\left[\max_{j \in [k]} \Big|\, |Q_j^{(L\ref{alg:before_while_1})}| - \alpha_j n\Big| > \frac{2\epsilon}{k} n\right] < e^{-\delta' n},
        \]
        where the last inequality is obtained by applying Claim~\ref{claim:linear_v} with $\epsilon'=2\epsilon/k$ to obtain $\delta', N'$.
        This yields the required result with $\delta = \delta'$ and $N = N'$.
    \end{claimproof}

    \begin{claim}\label{claim:linear_vm}
        After termination of the first while loop we have: 
        \begin{enumerate}
            \item 
                $|Q_j^{(L\ref{alg:after_while_1})}| = \alpha_j n$ for all $j \in [k]$,
            \item 
                for any $\epsilon > 0$, there exist $\delta>0$ and a natural number $N$, so that for every $n \ge N$
                \[
                    \Pr\!\left[ \max_{j \in [k]} \Big| S(Q_j^{(L\ref{alg:after_while_1})}) - \snk \Big| > \epsilon n^2 \right] < e^{-\delta n}.
                \]
                
        \end{enumerate}
        
    \end{claim}
    \begin{claimproof}
        The first claim follows as it is equivalent to the termination condition of the loop.
        The second claim follows from Claim~\ref{claim:linear_m} with $\epsilon', \delta'$ and $N'$ and from Claim~\ref{claim:linear_iter} with $\epsilon'', \delta''$ and $N''$. 
        Then, for $n \ge \max(N',N'')$,
        \begin{eqnarray*}
                \Pr\left[ \max_{j \in [k]} \Big| S(Q_j^{(L\ref{alg:before_while_1})}) - \snk \Big| > \epsilon' n^2 \right] &<& e^{-\delta' n}, \\
                \Pr[\iterI > \epsilon'' n] &<& e^{-\delta'' n}.
        \end{eqnarray*}
        As for any $j \in [k]$, each iteration of the first while loop changes $S(Q_j)$ by at most $n$, we conclude that $|S(Q_j^{(L\ref{alg:before_while_1})} - S(Q_j^{(L\ref{alg:after_while_1})}| \le \iterI \cdot n$. Therefore, performing a union bound on the bad events of claims \ref{claim:linear_m} and \ref{claim:linear_iter} with $\epsilon'=\epsilon''=\epsilon/2$, we conclude that
        \[
            \Pr\!\left[ \max_{j \in [k]} \Big| S(Q_j^{(L\ref{alg:after_while_1})}) - \snk \Big| > \epsilon n^2 \right]
                \le \Pr\!\left[ \max_{j \in [k]} \Big| S(Q_j^{(L\ref{alg:before_while_1})}) - \snk \Big| > \frac{\epsilon}{2} n^2 \right] + \Pr\!\Big[\iterI > \frac{\epsilon}{2} n\Big]
                < e^{-\delta' n} + e^{-\delta'' n},
        \]
        yielding the required result for any $0 < \delta < \min(\delta',\delta'')$ and sufficiently large $N$.
    \end{claimproof}

    \begin{claim}
        There are constant $N$ and $\epsilon' > 0$, so that the fluid mixing problem $\pi = \pi_n$ defined in the algorithm is $\epsilon'$-robust for all $n > N$.
    \end{claim}
    \begin{claimproof}
        We check the four conditions in the Definition~\ref{def:robust_fluid_mixing} for $\pi_n$ to be $\epsilon'$-robust, where $\pi_n$ corresponds to $P_n \in \PP_{\alpha_1,\ldots,\alpha_k}$. The total mass and total volume of $\pi_n$ are $M=n(n+1)/2$ and $V = n$:
        \begin{enumerate}
        \item 
            If $0 < \epsilon'/2 < \slack(\PP_{\alpha_1,\ldots,\alpha_k})$ 
            then $\epsilon'/2 < \lim_{n \rightarrow \infty} \slack(P_n)/n^2 = \frac{1}{2} \lim_{n \rightarrow \infty} \slack(\pi_n)/M$. 
            Therefore, $\slack(\pi_n)/M > \epsilon'$ for a sufficiently large $n$.
        \item If $0 < \epsilon' < \frac{1}{k}$, then $\frac{b_j}{M} = \frac{2\snk}{n(n+1)} = \frac{1}{k} > \epsilon' $ for all $j \in [k]$.
        \item If $0 < \epsilon' < \alpha_j$ for all $j \in [k]$, then $\frac{v_j}{V} = \frac{\alpha_j n }{n} > \epsilon'$ for all $j \in [k]$.
        \item If $0 < \epsilon' < \frac{1}{2}$, then  $\frac{M}{V} \frac{u_1}{a_1}= \frac{n+1}{2n}> \epsilon'$.
        \end{enumerate}
    \end{claimproof}

    \begin{claim}\label{claim:linear_pairs_before}
        For $j',j''\in[k]$ and $\Delta \in \N$, let $R_{j',j'',\Delta}^{(L\ref{alg:before_while_1})}$ be the random variable defined by
        \[
            R_{j',j'',\Delta}^{(L\ref{alg:before_while_1})} 
              = \Big|\big\{(z',z'') : z' - z'' = \Delta \textrm{ and } z' \in Q_{j'}^{(L\ref{alg:before_while_1})}  \textrm{ and } z'' \in Q_{j''}^{(L\ref{alg:before_while_1})}\big\}\Big|.
        \]
        Then, there exist constants $c, d, \eta > 0$ so that for a sufficiently large $n$
        \[
            \Pr\left[ \min_{j' \ne j'',\, \Delta \le d n} R_{j',j'',\Delta}^{(L\ref{alg:before_while_1})} < \eta n \right] < e^{-c n}.
        \]
    \end{claim}
    \begin{claimproof}
        Let $(x_{ij})$ be the solution for the fluid mixing problem $\pi$ defined in line~\ref{alg:problem_solution} of the algorithm.
        
        Then, as the problem $\pi$ is $\epsilon'$-robust for some $\epsilon' > 0$, by Lemma~\ref{lem:non_vanishing_last_block}, there exists $\delta > 0$ so that $x_{ij} > \delta$ for all $j \in [k]$ and $i > I_k$. Furthermore, $n-I_k \ge \alpha_k n$, since the $k$-th column of the matrix $(x_{ij})$ sums up to $\alpha_k n$, while $x_{ik} = 0$ for $i < I_k$ and $x_{ik} \le 1$ for all $i$.

        We set $0 < d < \alpha_k/3$ and $0 < \eta < \delta^2 d/2$ and for the rest of the proof, for the sake of brevity, we omit the $(L4)$ superscript.
        Then, given $j' \ne j''$ and $\Delta \le d n$ we express the corresponding bad event $R_{j',j'',\Delta} < \eta n$ as
        \[
            1_{\{R_{j',j'',\Delta} < \eta n\}} 
              = \sum_{i=1}^{n-\Delta} 1_{\{i \in Q_{j''}\}} 1_{\{i+\Delta \in Q_{j'}\}}
              \ge \sum_{\substack{I_k < i \le n-\Delta \\ i \bmod 2\Delta < \Delta}} 1_{\{i \in Q_{j''}\}} 1_{\{i+\Delta \in Q_{j'}\}},
        \]
        where $1_A$ denote the indicator function of the event $A$.
        Because of the constraint $i \bmod 2\Delta < \Delta$, we have $i + \Delta \bmod 2\Delta \ge \Delta$. Therefore the sets $\{i,i+\Delta\}$ are pairwise disjoint, implying that the terms of the last sum are independent Bernoulli$(p_i)$ random variables with $p_i = x_{i,j''}x_{i+\Delta,j'} > \delta^2$. Therefore, as the number of summands is at least $\frac{d n}{2}$, the last sum is lower bounded by a binomial$\big(\frac{d n}{2},\delta^2\big)$ random variable. Therefore, using Chernoff's bound, there exists some constant $c'>0$ so that
        \[
            \Pr[R_{j',j'',\Delta} < \eta n] < \exp[-c' n].
        \]
        Performing a union bound on $j',j''$ and $\Delta$ yields the required result, for $c = c'/2$. 
    \end{claimproof}
    
    \begin{claim}\label{claim:linear_pairs_after}
        For $j',j''\in[k]$ and $\Delta \in \N$, let $R_{j',j'',\Delta}^{(L\ref{alg:after_while_1})}$ be the random variable defined by
        \begin{equation}\label{eq:linear_rv2}
            R_{j',j'',\Delta}^{(L\ref{alg:after_while_1})} 
              = \Big|\big\{(z',z'') : z' - z'' = \Delta \textrm{ and } z' \in Q_{j'}^{(L\ref{alg:after_while_1})}  \textrm{ and } z'' \in Q_{j''}^{(L\ref{alg:after_while_1})}\big\}\Big|.
        \end{equation}
        Then, there exist constants $d',\eta', c' > 0$ so that for a sufficiently large $n$
        \[
            \Pr\left[ \min_{j' \ne j'',\, \Delta \le d' n} R_{j',j'',\Delta}^{(L\ref{alg:after_while_1})} < \eta' n \right] < e^{-c' n}.
        \]
    \end{claim}
    \begin{claimproof}
        As each iteration of the first while loop removes some element $i \in [n]$ from some $Q_{j'}$, it follows that the right hand side of \eqref{eq:linear_rv2} may decrease only for $R_{j',\cdot,\cdot}$ and by at most one. Therefore, $R^{(L\ref{alg:after_while_1})}_{j',j'',\Delta} \ge R^{(L\ref{alg:before_while_1})}_{j',j'',\Delta} - \iterI$.
        It follows by Claim~\ref{claim:linear_iter}, that for any $\epsilon > 0$, there exists $\delta >0$, so that for a sufficiently large $n$, we have $\Pr[\iterI > \epsilon n] < e^{-\delta n}$ and therefore
        \[
            \Pr\left[  R^{(L\ref{alg:after_while_1})}_{j',j'',\Delta} < R^{(L\ref{alg:before_while_1})}_{j',j'',\Delta} - \epsilon n \right] < e^{-\delta n}.
        \]        
        By Claim~\ref{claim:linear_pairs_before} there exist $d,\eta, c > 0$ so that for a sufficiently large $n$
        \[
            \Pr\left[ \min_{j' \ne j'',\, \Delta \le d n} R_{j',j'',\Delta}^{(L\ref{alg:before_while_1})} < \eta n \right] < e^{-c n}.        
        \]
        The required result follows by setting $\epsilon = \eta/2$, $d' = d$, $\eta' = \eta/2$, for any $0 < c' < \min(c,\delta)$.
    \end{claimproof}

    \begin{claim}\label{claim:second_while}
        Given any partition $Q_1 \sqcup \cdots \sqcup Q_k = [n]$, for $\Delta \in \N$ and $j', j'' \in [k]$ let
        \[
            R_{j',j'',\Delta} = \Big|\big\{(z',z'') : z' - z'' = \Delta \textrm{ and } z' \in Q_{j'}  \textrm{ and } z'' \in Q_{j''} \big\}\Big|. 
        \]        
        If for some value of $Z,D$, the following two equalities hold :
        \begin{eqnarray*}
            \max_{j \in [k]} |S(Q_j) - \snk| &\le& Z,\\
            \min_{j' \ne j'',\, \Delta \le D} R_{j',j'',\Delta} &>& \left\lceil \frac{k Z}{2 D} \right\rceil + 2 k \left\lceil \log_2 D \right\rceil,
        \end{eqnarray*}
        then applying the second while loop to the sets $Q_1, \ldots, Q_k$ terminates successfully.
    \end{claim}
    \begin{claimproof}
        Let $\iterII$ denote the number of iterations of the second while loop (until the successful completion or failure), for $j \in [k]$, let $Q^{(l)}_j$ denote the set $Q_j$ after $l$ iterations of the second while loop, where $0 \le l \le \iterII$, and let $R^{(l)}_{j',j'',\Delta}$ be defined accordingly.
        Let $y^{(l)}_j$ denote $S(Q^{(l)}_j) - \snk$ for all choices of $j,l$. Then, for any $j$, the sequence $y^{(0)}_j, y^{(1)}_j, \ldots, y^{(\iterII)}_j$ is monotone non-increasing in magnitude and does not cross zero. 
        We denote $J_P = \{j:y^{(0)}_j > 0\}$, $J_N = \{j:y^{(0)}_j < 0\}$, $y^{(l)}_P = \sum_{j \in J_P} y^{(l)}_j$ and $y^{(l)}_N = \sum_{j \in J_N} y^{(l)}_j$.
        Also, we denote $y^{(l)}_{\min} = \min_{j \in [k]} y^{(l)}_j$ and $y^{(l)}_{\max} = \max_{j \in [k]} y^{(l)}_j$.
        Then, $y^{(l)}_{\min} \le 0$ and $y^{(l)}_{\max} \ge 0$ for all $l$, with equality to zero only if $l = \iterII$ and the algorithm terminates successfully. 
        Finally, let $\delta^{(l)}$ be the difference $z'-z''$ in line~\ref{alg:2nd_while_z} of iteration $l$, where $l=0,\ldots,\iterII-1$.
        
        We divide the while loop iterations into stages, where stage $s$ starts at iteration $l_s$ and ends at iteration $l_{s+1}-1$,
        with $l_0 = 0$ and for $s \ge 1$,
        \[
            l_s = \min \left\{ l : y^{(l)}_{\max} \le \frac{D}{2^{s-1}} \textrm{ or } y^{(l)}_{\min} \ge -\frac{D}{2^{s-1}} \right\}.
        \] 
        Note that some of the stages may be empty, and that $l_s = \iterII$ for any $s$ such that $2^{s-1} > D$.
        
        Next, we bound $\iterII$ from above. We do this by upper-bounding the length of each stage and the number of stages.
        First, we claim that $l_1 \le \lceil \frac{k Z}{2 D} \rceil$.
        Indeed, we have 
        \[
            y^{(0)}_P - y^{(0)}_N = \sum_{j \in [k]} |y^{(0)}_j| \le k Z, \qquad\text{ and }\qquad
            y^{(0)}_P + y^{(0)}_N = 0.            
        \]
        Therefore, $y^{(0)}_P = -y^{(0)}_N \le k Z /2$.
        Since $\delta^{(l)} \ge D$, for any iteration $l$ in stage 1, the stage terminates after at most $\lceil \frac{k Z}{2D} \rceil$ steps.
        We now show that $l_{s+1}-l_s \le 2k$ for all $s \ge 1$. 
        By the definition of $l_s$, we have $\min\!\left( y^{(l_s)}_{\max}, -y^{(l_s)}_{\min}\right) \le D/2^{s-1}$, 
        implying that $y^{(l_s)}_P = -y^{(l_s)}_N \le (k-1)D/2^{s-1}$. 
        Since all iterations in stage $s$ have $\delta^{(l)} \ge D/2^s$, that stage terminates in at most $\left\lceil (\frac{(k-1)D}{2^{s-1}})/(\frac{D}{2^s}) \right\rceil \le 2k$ steps.

        As $l_s = \iterII$ for any $s \ge 1 + \lceil\log_2 D\rceil$, we conclude that
        \begin{equation}\label{eq:iterII_ub}
            \iterII \le \left\lceil \frac{k Z}{2 D} \right\rceil + 2 k \left\lceil \log_2 D \right\rceil.            
        \end{equation}
        As in iteration $l+1$ of the while loop, line~\ref{alg:2nd_while_z} swaps $z' \in Q^{(l)}_{j'}$ and $z'' \in Q^{(l)}_{j''}$, only $R^{(l)}_{j',\cdot,\cdot}$ and $R^{(l)}_{\cdot,j'',\cdot}$ may decrease and by at most one.
        It follows that under the claim assumptions, the while loop terminates before any of the $R^{(l)}_{\cdot,\cdot,\cdot}$ becomes zero.
    \end{claimproof}

    Finally, in order to prove Theorem~\ref{thm:alg_solves_linear_partition_whp}, we observe that by Claim~\ref{claim:linear_pairs_after}, there are constants $d',\eta', c' > 0$ so that for a sufficiently large $n$, we have
    \[
        \Pr\left[ \min_{j' \ne j'',\, \Delta \le d' n} R_{j',j'',\Delta}^{(L\ref{alg:after_while_1})} < \eta' n \right] < e^{-c' n}.
    \]
    Also, by Claim~\ref{claim:linear_vm}, for any $\epsilon > 0$, there exist $\delta>0$ so that for a sufficiently large $n$,
    \[
        \Pr\!\left[ \max_{j \in [k]} \Big| S(Q_j^{(L\ref{alg:after_while_1})}) - \snk \Big| > \epsilon n^2 \right] < e^{-\delta n}.
    \]
    Therefore, by Claim~\ref{claim:second_while} with $Z = \epsilon n^2$ and $D = d'n$, it follows that for a sufficiently large $n$ the algorithm terminates successfully with failure probability bounded by $e^{-c' n} + e^{-\delta n}$, if the following inequality holds:
    \[
        \eta' n > \left\lceil \frac{k Z}{2 D} \right\rceil + 2 k \left\lceil \log_2 D \right\rceil 
                = \left\lceil \frac{k \epsilon }{2 d'} n \right\rceil + 2 k \left\lceil \log_2(d'n) \right\rceil.
    \]
    As choosing any $0 < \epsilon < \frac{2 d' \eta'}{k}$ the inequality is satisfied for a sufficiently large $n$, yielding the required result.
\end{proof}

%%%%%%%%%%%%%%%%%%%%%%%%%%%%%%%%%%%%%%%%%%%%%%%%%%%%%%%%%%%%%%%%%%%%%%%%%%%%%%%%%%%%%%%%%%%%%%%%%%%%%%%%%%%
\section{Discussion}
Originally motivated by graph-labeling problems, this paper investigates Problem~\ref{prob:nk_partition}: the partitioning of $[n]$ into equal-sum sets of prescribed sizes. 
We term an instance of this problem \emph{ESPP-instance} (Equal Size Partition Problem). 

As demonstrated in Section~\ref{sec:fractional}, the \emph{slack condition}, which is the natural necessary condition for solvability of an ESPP-instance, is also a sufficient condition for the fractional relaxation of the problem. This provides a link to diverse fields, including operations research. 
While the slack condition is known to be sufficient to guarantee a solution when the number of parts is at most four~\cite{beena2009,kotlar16}, recent work \cite{ebrahem2024} has shown this does not hold in general; there exist instances of the problem that are not solvable, despite satisfying the slack condition~\cite{ebrahem2024}.

In this work, we approach this discrepancy from two different angles. First, we prove that the class of linear equal-sum partition problem families satisfying a strengthened version of the slack condition is solvable for sufficiently large n. Furthermore, we provide a randomized algorithm for finding such partitions that fails with only exponentially small probability. Second, we identify new non-solvable families of ESPP-instances, further clarifying the boundary of the problem.

Although a complete characterization remains elusive, and we conjecture the general problem to be NP-complete, our results suggest several promising paths forward. We conclude with a list of open problems and directions for future investigation:

\begin{itemize}
    \item All currently known examples of non-solvable ESPP instances contain (many) parts of size 2. Since our search has not produced so far non‑solvable ESPP instances such that the smallest part is of size at least 3, we propose two contrasting avenues for further investigation: either prove that ESPP instances are solvable when the partition contains no parts of size 2, or develop criteria for non‑solvability that apply to such instances, and for which the corresponding family of examples is non‑empty. The criteria introduced in \cite{ebrahem2024} for partitions whose smallest parts have size 2 may offer useful insight in this direction.

    \item Define conditions that would constitute a sufficient set of conditions for solvability, or

    \item Demonstrate that the general version of Problem~\ref{prob:nk_partition} is NP‑complete. The two criteria established in \cite{ebrahem2024}, together with the additional criterion in Subsection~\ref{sec:criterion}, indicate that the problem is more intricate than previously assumed. In particular, they highlight the inadequacy of the earlier, overly optimistic conjecture that the slack condition alone would suffice.

    \item Extend the result of Section~\ref{sec:linear} to other families of partitions. For example, when $k$ grows with $n$, or when the sizes of the parts are not linear in $n$ but rather $o(n)$.

    \item Generalize Problem~\ref{prob:nk_partition} and the slack condition to an arbitrary set of numbers, instead of just $[n]$, or to the case where the sums of the parts are some given values that are not necessarily equal. For the latter case, the situation involving undetermined part sizes was treated in \cite{chen2005}.
\end{itemize}

\section*{Acknowledgements}

The authors would like to thank 
Boris Gurevich for writing code for searching for counterexamples, and Abthal Heeb for writing the code that tested the algorithms presented in this work.

%%%%%%%%%%%%%%%%%%%%%%%%%%%%%%%%%%%%%%%%%%%%%%%%%%%%%%%%%%%%%%%%%%%%%%%%%%%%%%%%%%%%%%%%%%%%%%%%%%%%%%%%%%%

\bibliographystyle{abbrv}
\bibliography{refs}

\newpage
\appendix
\section{Iterative algorithm for solving the fluid mixing problem}\label{app:iterative}

\begin{algorithm}[H]
\caption{Iterative Solution For The Fluid Mixing Problem}\label{alg:fluid_mixing_solution}
\begin{algorithmic}[1]
    \Require A fluid mixing problem $\pi=(\overline{a}, \overline{u}, \overline{b}, \overline{v})$ with $n$ source containers and $k$ target containers with $\slack(\pi) \ge 0$
    \Ensure  A solution $X = \{x_{i,j}\}_{i\in[n], j\in[k]}$ for $\pi$ 
    \State $X \gets 0_{n \times k}$, 
    \State $i \gets 1$, $j \gets 1$, $B \gets b_1$, $V \gets v_1$,  $\overline{w} \gets \delta_1 \in \R_{\ge0}^k$
    \While{$\|u\|_1 > 0$}
        \If{V = 0} \Comment{Implies that $B=0$, $j<k$ and $\slack_j(\pi) = 0$}
            \State $j \gets j+1$
            \State $V = v_j$, $B \gets b_j$, $\overline{w} \gets \delta_j$
        \EndIf
        \If{$u_i = 0$}
            \State $i \gets i+1$
        \EndIf
        \While{$j<k$ and $B/V = b_{j+1}/v_{j+1}$}
            \State $j \gets j+1$
            \State $\gamma \gets V/(V+v_j)$
            \State $\overline{w} \gets \gamma \overline{w} + (1-\gamma)\delta_j$
            \State $V \gets V + v_j$, $B \gets B + b_j$
        \EndWhile
        \State $\lambda \gets \min(1,V/u_i)$
        \If{$j<k$ and $\slack_j(\pi) > 0$}
            \State $\lambda \gets \min(\lambda,\frac{B v_{j+1} \,-\, V b_{j+1}}{a_i v_{j+1} \,-\, u_i b_{j+1}})$
        \EndIf
        \State $a_i \gets a_i - \lambda a_i$, $u_i \gets u_i - \lambda u_i$
        \State $B \gets B - \lambda a_i$, $V \gets V - \lambda u_i$
        \State $X_{i,\cdot} \gets X_{i,\cdot} + \lambda \overline{w}$
    \EndWhile
\end{algorithmic}
\end{algorithm}
\end{document}